\documentclass[12pt]{amsart}
\calclayout
\usepackage{amssymb}
\usepackage{amsmath}
\usepackage{graphics}
\usepackage{epsfig, color}
\newcommand\R{\mathbb R}

\newcommand\C{\mathcal C}
\newcommand\K{\mathcal K}
\renewcommand\O{\mathcal O}
\newcommand\T{\mathcal T}
\newcommand\E{\mathcal E}
\newcommand\M{\mathcal M}

\newcommand\N{\mathcal N}
\newcommand\V{\mathcal V}

\renewcommand\H{\mathcal H}
\renewcommand\P{\mathcal P}
\renewcommand\div{\operatorname{div}}

\newcommand\eps{\operatorname{\epsilon}}
\newcommand\veps{\operatorname{\varepsilon}}
\newcommand\vtheta{\operatorname{\vartheta}}

\newcommand\x{\times}
\renewcommand\t{\tilde}
\newcommand\lbra{[\![}
\newcommand\rbra{]\!]}
\newcommand\lbrac{\,[\!\!\!\{}
\newcommand\rbrac{\}\!\!\!]\,}

\renewcommand\ll{|\kern-2pt|\kern-2pt|}

\newcommand\e{^\epsilon}

\numberwithin{equation}{section}
\theoremstyle{plain}
\newtheorem{thm}{Theorem}

\newtheorem{lem}[thm]{Lemma}

\numberwithin{thm}{section}

\theoremstyle{remark}
\newtheorem*{remark}{Remark}

\def\underput#1#2#3{
\mathchoice
{\vtop{\ialign{##\crcr\hfil$#2$\vrule width0pt height0pt depth#3\hfil\crcr
\noalign{\nointerlineskip}\hfil$\scriptstyle#1$\hfil\crcr}}}
{\vtop{\ialign{##\crcr\hfil$#2$\vrule width0pt height0pt depth#3\hfil\crcr
\noalign{\nointerlineskip}\hfil$\scriptstyle#1$\hfil\crcr}}}
{\vtop{\ialign{##\crcr\hfil$\scriptstyle#2$\vrule width0pt height0pt
depth#3\hfil\crcr
\noalign{\nointerlineskip}\hfil$\scriptscriptstyle#1$\hfil\crcr}}}
{\vtop{\ialign{##\crcr\hfil$\scriptscriptstyle#2$\vrule width0pt height0pt
depth#3\hfil\crcr
\noalign{\nointerlineskip}\hfil$\scriptscriptstyle#1$\hfil\crcr}}}}

\def\stack#1#2#3{\rlap{#1}\lower#3\hbox{#2}}

\def\twiddlespace{1.2truept}

\def\dtwiddle{\displaystyle\sim}
\def\ttwiddle{\textstyle\sim}
\def\stwiddle{\scriptstyle\sim}
\def\sstwiddle{\scriptscriptstyle\sim}

\def\doubledtwiddle{\stack{$\dtwiddle$}{$\dtwiddle$}{\twiddlespace}}
\def\doublettwiddle{\stack{$\ttwiddle$}{$\ttwiddle$}{\twiddlespace}}
\def\doublestwiddle{\stack{$\stwiddle$}{$\stwiddle$}{\twiddlespace}}
\def\doublesstwiddle{\stack{$\sstwiddle$}{$\sstwiddle$}{\twiddlespace}}

\def\tripledtwiddle{\stack{$\dtwiddle$}{$\doubledtwiddle$}{\twiddlespace}}
\def\triplettwiddle{\stack{$\ttwiddle$}{$\doublettwiddle$}{\twiddlespace}}
\def\triplestwiddle{\stack{$\stwiddle$}{$\doublestwiddle$}{\twiddlespace}}
\def\triplesstwiddle{\stack{$\sstwiddle$}{$\doublesstwiddle$}{\twiddlespace}}

\def\quadrupledtwiddle{\stack{$\dtwiddle$}{$\tripledtwiddle$}{\twiddlespace}}
\def\quadruplettwiddle{\stack{$\ttwiddle$}{$\triplettwiddle$}{\twiddlespace}}
\def\quadruplestwiddle{\stack{$\stwiddle$}{$\triplestwiddle$}{\twiddlespace}}
\def\quadruplesstwiddle{\stack{$\sstwiddle$}{$\triplesstwiddle$}{\twiddlespace}}\def\quadru
pletwiddle{{\mathchoice{\quadrupledtwiddle}{\quadruplettwiddle}%

{\quadruplestwiddle}{\quadruplesstwiddle}}}

\def\strikedist{3pt}

\def\dstrike{\vrule width7pt height0pt depth.4pt}
\def\tstrike{\vrule width7pt height0pt depth.4pt}
\def\sstrike{\hbox{\vrule width5pt height0pt depth.4pt}}
\def\ssstrike{\hbox{\vrule width3pt height0pt depth.4pt}}
\def\strike{{\mathchoice{\dstrike}{\tstrike}{\sstrike}{\ssstrike}}}

\def\ub#1{\underput\strike{#1}{\strikedist}}

\catcode`\!=11
\newcommand\!a{{\boldsymbol a}}
\newcommand\!b{{\boldsymbol b}}
\newcommand\!c{{\boldsymbol c}}
\newcommand\!d{{\boldsymbol d}}
\newcommand\!e{{\boldsymbol e}}
\newcommand\!f{{\boldsymbol f}}
\newcommand\!g{{\boldsymbol g}}
\newcommand\!h{{\boldsymbol h}}
\newcommand\!i{{\boldsymbol i}}
\newcommand\!j{{\boldsymbol j}}
\newcommand\!k{{\boldsymbol k}}
\newcommand\!l{{\boldsymbol l}}
\newcommand\!m{{\boldsymbol m}}
\newcommand\!n{{\boldsymbol n}}
\newcommand\!o{{\boldsymbol o}}
\newcommand\!p{{\boldsymbol p}}
\newcommand\!q{{\boldsymbol q}}
\newcommand\!r{{\boldsymbol r}}
\newcommand\!s{{\boldsymbol s}}
\newcommand\!t{{\boldsymbol t}}
\newcommand\!u{{\boldsymbol u}}
\newcommand\!v{{\boldsymbol v}}
\newcommand\!w{{\boldsymbol w}}
\newcommand\!x{{\boldsymbol x}}
\newcommand\!y{{\boldsymbol y}}
\newcommand\!z{{\boldsymbol z}}
\newcommand\!A{{\boldsymbol A}}
\newcommand\!B{{\boldsymbol B}}
\newcommand\!C{{\boldsymbol C}}
\newcommand\!D{{\boldsymbol D}}
\newcommand\!E{{\boldsymbol E}}
\newcommand\!F{{\boldsymbol F}}
\newcommand\!G{{\boldsymbol G}}
\newcommand\!H{{\boldsymbol H}}
\newcommand\!I{{\boldsymbol I}}
\newcommand\!J{{\boldsymbol J}}
\newcommand\!K{{\boldsymbol K}}
\newcommand\!L{{\boldsymbol L}}
\newcommand\!M{{\boldsymbol M}}
\newcommand\!N{{\boldsymbol N}}
\newcommand\!O{{\boldsymbol O}}
\newcommand\!P{{\boldsymbol P}}
\newcommand\!Q{{\boldsymbol Q}}
\newcommand\!R{{\boldsymbol R}}
\newcommand\!S{{\boldsymbol S}}
\newcommand\!T{{\boldsymbol T}}
\newcommand\!U{{\boldsymbol U}}
\newcommand\!V{{\boldsymbol V}}
\newcommand\!W{{\boldsymbol W}}
\newcommand\!X{{\boldsymbol X}}
\newcommand\!Y{{\boldsymbol Y}}
\newcommand\!Z{{\boldsymbol Z}}
\newcommand\!alpha{{\boldsymbol\alpha}}
\newcommand\!beta{{\boldsymbol\beta}}
\newcommand\!gamma{{\boldsymbol\gamma}}
\newcommand\!delta{{\boldsymbol\delta}}
\newcommand\!epsilon{{\boldsymbol\epsilon}}
\newcommand\!zeta{{\boldsymbol\zeta}}
\newcommand\!eta{{\boldsymbol\eta}}
\newcommand\!theta{{\boldsymbol\theta}}
\newcommand\!iota{{\boldsymbol\iota}}
\newcommand\!kappa{{\boldsymbol\kappa}}
\newcommand\!lambda{{\boldsymbol\lambda}}
\newcommand\!mu{{\boldsymbol\mu}}
\newcommand\!nu{{\boldsymbol\nu}}
\newcommand\!xi{{\boldsymbol\xi}}
\newcommand\!pi{{\boldsymbol\pi}}
\newcommand\!rho{{\boldsymbol\rho}}
\newcommand\!sigma{{\boldsymbol\sigma}}
\newcommand\!tau{{\boldsymbol\tau}}
\newcommand\!upsilon{{\boldsymbol\upsilon}}
\newcommand\!phi{{\boldsymbol\phi}}
\newcommand\!chi{{\boldsymbol\chi}}
\newcommand\!psi{{\boldsymbol\psi}}
\newcommand\!omega{{\boldsymbol\omega}}
\newcommand\!varepsilon{{\boldsymbol\varepsilon}}
\newcommand\!vartheta{{\boldsymbol\vartheta}}
\newcommand\!varpi{{\boldsymbol\varpi}}
\newcommand\!varrho{{\boldsymbol\varrho}}
\newcommand\!varsigma{{\boldsymbol\varsigma}}
\newcommand\!varphi{{\boldsymbol\varphi}}
\newcommand\!Gamma{{\boldsymbol\Gamma}}
\newcommand\!Delta{{\boldsymbol\Delta}}
\newcommand\!Theta{{\boldsymbol\Theta}}
\newcommand\!Lambda{{\boldsymbol\Lambda}}
\newcommand\!Xi{{\boldsymbol\Xi}}
\newcommand\!Pi{{\boldsymbol\Pi}}
\newcommand\!Sigma{{\boldsymbol\Omega\eigma}}
\newcommand\!Upsilon{{\boldsymbol\Upsilon}}
\newcommand\!Phi{{\boldsymbol\Phi}}
\newcommand\!Psi{{\boldsymbol\Psi}}
\newcommand\!Omega{{\boldsymbol\Omega}}

\begin{document}

\title [Linear finite element for Naghdi shell]
{A linear finite element procedure \\for the Naghdi shell model}


\author{Sheng Zhang}

\thanks{Department of Mathematics, Wayne State University, Detroit, MI 48202 (\texttt{szhang@wayne.edu})}

\begin{abstract}
We prove the accuracy of a mixed finite element method for bending dominated shells in which a major part of the membrane/shear strain is
reduced, to free up membrane/shear locking.
When no part of the membrane/shear strain is reduced, the method becomes a
consistent discontinuous Galerkin method that is proven accurate for 
membrane/shear dominated shells and intermediate shells. The two methods can be coded in a single program by using a parameter.
We propose a procedure of numerically detecting the asymptotic behavior of a shell,
choosing the parameter value in the method, and
producing accurate approximation for a given shell problem.
The method uses piecewise linear functions to approximate all the variables.
The analysis is carried out for shells whose middle surfaces have the most general geometries,
which shows that the method has the optimal order of accuracy for general shells and the accuracy is robust with respect to the shell
thickness.
In the particular case that  the  geometrical coefficients 
of the
shell middle surface are piecewise constants the
accuracy is uniform with
respect to the shell thickness.

\vspace{12pt}

\end{abstract}
\maketitle


\section{Introduction}

In the Naghdi shell model, the strain energy is a sum of bending, membrane, and transverse shear strain energies, in which the bending part
and the membrane/shear part scale differently with respect to the shell thickness.
Depending on the shell shape, loading force, and boundary conditions,  a shell
deformation could be bending dominated, membrane/shear dominated, or intermediate.
These asymptotic behaviors are very different from each other. It has been unsuccessful to
find a single finite element model that is provably accurate and reliable, which is
capable of accurately simulating shell behaviors in all the asymptotic regimes.
We propose a finite element procedure
for numerical computation of shell mechanics, in which we only need to write one program involving
an adjustable parameter which can accurately compute
shells of any kind of asymptotic behaviors.
The method is primarily designed for bending dominated shells, with an intention of being
useful to membrane/shear dominated and intermediate shell deformations.
Setting the program parameter to different values, we get programs for two different methods,
one of which is suitable for bending dominated shell deformations  and the other
is for membrane/shear dominated and intermediate shells.
We prove that the methods produce accurate results for bending and non-bending dominated shells, respectively,
and propose a procedure for properly using the methods to obtain good numerical results for any given shell problem.
Proving the theorems on error estimates for the two methods constitutes the major part of this paper.

The  first method is accurate for bending dominated shells, but may not be accurate for membrane/shear dominated shells or intermediate shells.
It is a partially projected mixed finite element method, in which a major part of the membrane/shear energy is reduced.
In the method, the primary variables, including the tangential displacement vector, transverse displacement scalar, and normal fiber rotation vector,
are approximated by discontinuous piecewise linear functions on all elements but those that have edges belong to the portion of shell boundary
on which the shell is free or subject to stress boundary condition. 
On elements that have one edge on the free boundary, for stability purpose, we need to add two degree of freedoms to a linear displacement approximation.
It is sufficient to approximate the displacement
by quadratic functions. On elements that have two edges on the shell free boundary, the stability requires four additional degree of freedoms, and
one may use cubic functions to do the
approximation.
The secondary membrane stress tensor and transverse shear stress vector are approximated
by continuous piecewise linear functions.
Peculiar to shell problems, it seems that generally there is no absolute stability holding
uniformly with respect to the shell thickness. The curved nature of the shell
midsurface imposes an additional condition on the finite element triangulation that the mesh needs to be properly refined where the geometrical coefficients
changes rapidly, and the  refinement is related to the shell thickness.
Although the discontinuous approximation to the primary variables
renders the indispensable flexibility, it still can not completely assure the usual inf-sup condition for mixed methods, which  needs to be compromised slightly.
As a consequence, in the error estimate a factor of the form
\begin{equation*}
1+\veps^{-1}
\max_{\tau\in\T_h}h^2_\tau\left(\sum_{\alpha,\beta,\lambda=1,2}|\Gamma^{\lambda}_{\alpha\beta}|_{1,\infty,\tau}+
\sum_{\alpha,\beta=1,2}|b_{\alpha\beta}|_{1,\infty,\tau}+\sum_{\alpha,\beta=1,2}|b^\beta_\alpha|_{1,\infty,\tau}\right)
\end{equation*}
is multiplied to the norm of error of optimal approximation whose order is $\O(h)$.
Here, $\T_h$ is a
finite element triangulation of the parameter domain of the shell midsurface, which
is assumed to be shape regular, but not necessarily quasi-uniform.
Therefore,  optimal accuracy requires the above quantity to be bounded by a constant $C$.
Throughout the paper, we use $C$ to denote a constant that is independent of $\veps$. It may depend on the
shape regularity $\K$ of the triangulation, but otherwise independent of the finite element mesh.
We will simply say that such a constant is independent of  the triangulation.
When the curvature tensors $b_{\alpha\beta}$, $b^{\alpha}_\beta$, and the Christoffel symbols $\Gamma^{\lambda}_{\alpha\beta}$ are
piecewise constant, this extra condition is trivially satisfied, which is the case,
for example, when the shell is circular cylindrical and parameterized by cylindrical coordinates.

The second method is accurate for shell problems that
are  membrane/shear dominated or intermediate.  The finite element method  is essentially a consistent discontinuous Galerkin method
for a singular perturbation problem, which only involves the primary variables. In the method all the variables are approximated by
discontinuous piecewise linear functions, and no enrichment for displacement on elements attached to the shell free boundary
should be included, which is necessary for stability of the method for bending dominated shells but could thwart the stability and well posedness of the finite element model for
shell deformations  that are not bending dominated.
In this case,
the well-posedness of the finite element model can not be assured by a ``sufficiently big  penalty'' on the discontinuity because of the weakness of the membrane/shear energy norm.
We prove that if the triangulation satisfies the condition that
\begin{equation*} 
\max_{\tau\in\T_h}h^2_\tau\left(\sum_{\alpha,\beta,\lambda,\delta=1,2}|\Gamma^{\lambda}_{\alpha\beta}|_{\delta,\infty,\tau}+
\sum_{\alpha,\beta,\delta=1,2}|b_{\alpha\beta}|_{\delta,\infty,\tau}+\sum_{\alpha,\beta,\delta=1,2}|b^\beta_\alpha|_{\delta,\infty,\tau}\right)
\le C\veps
\end{equation*}
then when the penalty constant, that is independent of the shell thickness and triangulation,
on the discontinuity is sufficiently big, the finite element model is well defined, and has the optimal order of accuracy
in the relative energy norm for intermediate and membrane/shear dominated shell deformations.
When this method is applied to bending dominated shells, the numerical membrane/shear locking could significantly
undermine its accuracy.

For many shell problems, one can determine their asymptotic  regimes {\em a priori} by partial differential equation theories.
For example, a totally clamped elliptic shell must be membrane/shear dominated no matter how it is loaded;
a partially clamped spherical shell is always intermediate;  a cylindrical shell clamped only on its edge of rulings and subject to a transverse force
is bending dominated, etc. \cite{Bathe-book, CiarletIII, Sanchez}.  We can also use the finite element methods to numerically detect the asymptotic behavior of a shell  and determine
its asymptotic regime, and thus determine which method is suitable for the given shell
problem. This is  based on some results of shell asymptotic analysis and the accuracy of finite element methods to be proved below.
We describe a procedure of numerical detection  at the end of this paper.

The paper is organized as follows. In Section~\ref{SHELL} we recall the shell model and its equivalent formulations. In this section we also include an
asymptotic analysis on an abstract level, from which we derive some asymptotic estimates on the shell model solution.
These asymptotic estimates will be needed to
interpret our results of numerical analysis.  They also furnish the basis for the mechanism we use to detect the type of shell deformations.
In Section~\ref{FEMmodel} we define the finite element models and the finite element spaces.
Section~\ref{KornOnShell} contains a discrete Korn type inequality for Naghdi shell, which is crucial for numerical analysis.
In Section~\ref{BendingErrorAnalysis} we analyze the accuracy of the mixed finite element  model when the shell problem is bending dominated.
Section \ref{MembraneErrorAnalysis}
is devoted to analysis of the discontinuous Galerkin method for membrane/shear dominated shells and intermediate shells. Finally, in Section~\ref{Procedure}, we
describe a procedure of choosing the suitable method for a given shell problem.
We use superscripts and subscripts to indicate contravariant and covariant components of vectors and tensors.
Greek scripts take their values in $\{1, 2\}$, while Latin scripts take their values in $\{1, 2, 3\}$.
Summation rules with respect to repeated sup and subscripts will also be used.
We use $A\lesssim B$ to represent $A\le CB$, and $A\simeq B$ means $A\lesssim B$ and $B\lesssim A$.

\section{The shell model and its asymptotic behavior}
\label{SHELL}
\subsection{The Naghdi shell model}
Let $\t\Omega\subset\R^3$
be the middle surface of a shell of thickness $2\veps$.
It is  the image of a domain
$\Omega\subset\R^2$ through a mapping $\!Phi$.
The coordinates $x_\alpha\in\Omega$
then furnish the curvilinear coordinates on $\t\Omega$.
We assume that at any point on the surface,
along the coordinate lines,
the two tangential vectors
$\!a_{\alpha}={\partial\!Phi}/{\partial x_{\alpha}}$
are linearly independent.
The unit vector
$\!a_3=(\!a_1\x\!a_2)/|\!a_1\x\!a_2|$ is normal to $\t\Omega$.
The triple $\!a_i$ furnishes the covariant basis on $\t\Omega$.
The contravariant basis
$\!a^i$ is defined by the relations
$\!a^{\alpha}\cdot\!a_{\beta}=\delta^{\alpha}_{\beta}$ and $\!a^3=\!a_3$,
in which $\delta^{\alpha}_{\beta}$ is the Kronecker delta.
The metric tensor is defined in terms of its  covariant components
by $a_{\alpha\beta}=\!a_{\alpha}\cdot\!a_{\beta}$,  the determinant of
which is denoted by $a$. The contravariant components
are given by
$a^{\alpha\beta}=\!a^{\alpha}\cdot\!a^{\beta}$.
The curvature tensor
has covariant components
$b_{\alpha\beta}=\!a_3\cdot\partial_{\beta}\!a_{\alpha}$, whose
mixed components are $b^{\alpha}_{\beta}=a^{\alpha\gamma}b_{\gamma\beta}$.
The symmetric tensor $c_{\alpha\beta}=b^\gamma_\alpha b_{\gamma\beta}$ is called the third
fundamental form of the surface.
The Christoffel symbols
are defined by
$\Gamma^{\gamma}_{\alpha\beta}
=\!a^{\gamma}\cdot\partial_{\beta}\!a_{\alpha}$,
which are symmetric with respect to the subscripts. The derivative of a scalar is a covariant vector.
The covariant derivative of a vector or tensor is a higher order tensor.
The formulas below will be used.
\begin{equation}\label{covariant-derivative}
\begin{gathered}
u_{\alpha|\beta}=\partial_{\beta}u_{\alpha}-\Gamma^{\gamma}_{\alpha\beta}
u_{\gamma},\quad
\eta^\alpha|_\beta=\partial_\beta\eta^\alpha+\Gamma^\alpha_{\beta\delta}\eta^\delta,\\
\sigma^{\alpha\beta}|_{\gamma}=\partial_{\gamma}\sigma^{\alpha\beta}
+\Gamma^{\alpha}_{\gamma\lambda}\sigma^{\lambda\beta}
+\Gamma^{\beta}_{\gamma\tau}\sigma^{\alpha\tau}.
\end{gathered}
\end{equation}
Product rules for covariant differentiation, like
$(\sigma^{\alpha\lambda}u_{\lambda})|_{\beta}=
\sigma^{\alpha\lambda}|_{\beta}u_{\lambda}
+\sigma^{\alpha\lambda}u_{\lambda|\beta}$,
are valid. For more information see \cite{GZ}.

In a shell deformation, the middle surface displacement is a vector field $u_\alpha\!a^\alpha+w\!a^3$, and normal fiber rotation
is represented by a vector field $\theta_\alpha\!a^\alpha$.
The Naghdi shell model \cite{ABrezzi2, Bathe-book,  CiarletIII,  Naghdi} uses the displacement components $u_\alpha$, $w$, and
the rotation components $\theta_\alpha$ as the primary variables.
The bending strain tensor, membrane strain tensor, and transverse shear strain vector due to the deformation represented by
such a set of primary variables are
\begin{equation}\label{N-bending}
\rho_{\alpha\beta}(\!theta, \!u, w)=
\frac12(\theta_{\alpha|\beta}+\theta_{\beta|\alpha})-\frac12(b^\gamma_\alpha u_{\gamma|\beta}+b^\gamma_\beta u_{\gamma|\alpha})+c_{\alpha\beta}w,
\end{equation}
\begin{equation}\label{N-metric}
\gamma_{\alpha\beta}(\!u,w)=
\frac12(u_{\alpha|\beta}+u_{\beta|\alpha})
-b_{\alpha\beta}w,
\end{equation}
\begin{equation}\label{N-shear}
\tau_\alpha(\!theta, \!u, w)=\partial_\alpha w+b^\gamma_\alpha u_\gamma+\theta_\alpha.
\end{equation}
Here we used $\!theta$ and $\!u$ to represent the vectors $\theta_\alpha$ and $u_\alpha$, respectively.

Let the boundary $\partial\t\Omega$ be divided to $\partial^D\t\Omega\cup\partial^S\t\Omega\cup\partial^F\t\Omega$.
On $\partial^D\t\Omega$ the shell is clamped, on $\partial^S\t\Omega$ the shell is soft-simply supported, and
on $\partial^F\t\Omega$ the shell is free of displacement constraint and subject to force or moment  only.
There are $32$ different ways to specify boundary conditions at any point on the shell boundary, of which we consider the three most typical.
The shell model is
defined in the Hilbert space
\begin{multline}\label{N-space}
H=\{(\!phi, \!v, z)\in \!H^1\x\!H^1\x H^1;\  \phi_\alpha, v_\alpha \text{ and } z \text{ are }0\  \text{on}\ \partial^D\Omega, \\
\text{ and  }v_\alpha\text{ and } z \text{ are }0\
\text{on}\ \partial^S\Omega\}.
\end{multline}
The model determines a unique $(\!theta, \!u, w)\in H$
such that
\begin{multline}\label{N-model}
\frac13\int_{\t\Omega}
a^{\alpha\beta\lambda\gamma}\rho_{\lambda\gamma}(\!theta, \!u, w)
\rho_{\alpha\beta}
(\!phi, \!v, z)\\
+\veps^{-2}\int_{\t\Omega}
a^{\alpha\beta\lambda\gamma}\gamma_{\lambda\gamma}(\!u, w)
\gamma_{\alpha\beta}(\!v,z)+\kappa\mu\veps^{-2}\int_{\t\Omega}a^{\alpha\beta}\tau_\alpha(\!theta, \!u, w)\tau_\beta(\!phi, \!v, z)
\\
\hfill =
\int_{\t\Omega}
(p^{\alpha}v_{\alpha}+
p^3z)
+\int_{\partial^S\t\Omega}r^\alpha\phi_\alpha
+\int_{\partial^F\t\Omega}\left(q^\alpha v_\alpha+q^3z+r^\alpha\phi_\alpha\right)
\\
\hfill \forall\
(\!phi, \!v,z) \in H.\\
\end{multline}
Here, $p^i$ is the resultant loading force density on the shell mid surface, and  $q^i$ and $r^\alpha$ are the force resultant and
moment resultant on the shell edge \cite{Naghdi}.
The factor $\kappa$ is a shear correction factor.
The fourth order contravariant tensor
$a^{\alpha\beta\gamma\delta}$ is the elastic tensor of the shell,
defined by
\begin{equation*}
a^{\alpha\beta\gamma\delta}=\mu (a^{\alpha\gamma}a^{\beta\delta}+a^{\beta\gamma}a^{\alpha\delta})+
\frac{2\mu\lambda}{2\mu+\lambda}
a^{\alpha\beta}a^{\gamma\delta}.
\end{equation*}
Here, $\lambda$ and $\mu$ are the Lam\'e coefficients of the elastic material, which are assumed to be constant.
The compliance tensor of the shell defines the inverse operator of the elastic tensor, given by
\begin{equation*}
a_{\alpha\beta\gamma\delta}=\frac{1}{2\mu}\left[
\frac12(a_{\alpha\delta}a_{\beta\gamma}+
a_{\beta\delta}a_{\alpha\gamma})-\frac{\lambda}{2\mu+3\lambda}a_{\alpha\beta}a_{\gamma\delta}
\right].
\end{equation*}
The integrals on $\t\Omega$ are taken with respect to the surface area element. A function $f$ defined on $\Omega$ is identified with a function
on $\t\Omega$ through the mapping $\!Phi$ and denoted by the same symbol.
Thus $f(x_\alpha)=f(\!Phi(x_\alpha))$, and we have $\int_{\t\Omega}f=\int_\Omega f(x_\alpha)\sqrt a dx_1dx_2$.
The integrals on the curves $\partial^S\t\Omega$ and $\partial^F\t\Omega$ are taken with respect to their arc length. Similar notations will be used
in the following for integrals on subregions of $\t\Omega$ and any curves in $\t\Omega$.
Generally, we use a tilde to indicate an object or operation on the shell midsurface $\t\Omega$. The same notation
without tilde represents the corresponding object or operation on the coordinate domain $\Omega$.
The model  \eqref{N-model} has a unique solution in the space $H$ \cite{BCM, CiarletIII}. When $\veps\to 0$, its solution
behaves in very different manners. Depending on the shell shape, loading force, and boundary conditions,
the shell deformation could be bending dominated, membrane/shear dominated, or intermediate \cite{Bathe-book, Sanchez}.
If the resultant loading functions $p^i$, $q^i$, and $r^\alpha$ are independent of $\veps$,
when $\veps\to 0$, the  model solution converges to a nonzero limit that solves a limiting bending model  for bending dominated shell problems.
In contrast, the model solution tends to zero for shell deformations that are not bending dominated.
In the latter case, to observe the higher order response from the shell,
we could scale the loading functions by multiplying them with $\veps^{-2}$.

For $(\!theta, \!u, w)$ and $(\!phi, \!v, z)$ in $H$, we define the following bilinear forms corresponding to the bending, membrane, and transverse
shear strain energies, respectively.
\begin{equation}\label{bilinear-forms}
\begin{gathered}
\rho(\!theta, \!u, w; \!phi, \!v, z)=
\frac13\int_{\t\Omega}
a^{\alpha\beta\lambda\gamma}\rho_{\lambda\gamma}(\!theta, \!u, w)
\rho_{\alpha\beta}(\!phi, \!v, z),
\\
\gamma(\!u, w; \!v, z)=\int_{\t\Omega}
a^{\alpha\beta\lambda\gamma}\gamma_{\lambda\gamma}(\!u,w)
\gamma_{\alpha\beta}(\!v,z),
\\
\tau(\!theta, \!u, w; \!phi, \!v, z)=\kappa\mu\int_{\t\Omega}a^{\alpha\beta}\tau_\alpha(\!theta, \!u, w)\tau_\beta(\!phi, \!v, z).
\end{gathered}
\end{equation}
We let
\begin{equation}\label{bilinear-form-a}
a(\!theta, \!u, w; \!phi, \!v, z)=\rho(\!theta, \!u, w; \!phi, \!v, z)+\gamma(\!u, w; \!v, z)+\tau(\!theta, \!u, w; \!phi, \!v, z).
\end{equation}
We also define a linear form to represent the loading functional
\begin{equation}\label{form_f}
\langle \!f; \!phi, \!v, z\rangle=\int_{\t\Omega}
(p^{\alpha}v_{\alpha}+
p^3z)
+\int_{\t\E^S}r^\alpha\phi_\alpha
+\int_{\t\E^F}\left(q^\alpha v_\alpha+q^3z+r^\alpha\phi_\alpha\right).
\end{equation}

As did in \cite{ABrezzi2, Bramble-Sun2, Suri} and the references cited in these papers,
we split the membrane and shear parts in the model by writing $\veps^{-2}=\eps^{-2}+1$, and write the model in an equivalent form.
It seeks $(\!theta\e, \!u\e, w\e)\in H$ such that
\begin{multline}\label{N-0-model}
a(\!theta\e, \!u\e, w\e; \!phi, \!v, z)
+\eps^{-2}\left[
\gamma(\!u\e, w\e; \!v,z)+\tau(\!theta\e, \!u\e, w; \!phi, \!v, z)\right]\\=
\langle \!f; \!phi, \!v, z\rangle\ \ \forall\
(\!phi, \!v,z) \in H.
\end{multline}

We introduce the membrane stress tensor
$\M^{\eps\alpha\beta}=
\eps^{-2}a^{\alpha\beta\lambda\gamma}\gamma_{\lambda\gamma}(\!u\e, w\e)$
and the shear stress vector $\xi^{\eps\alpha}=\eps^{-2}\kappa\mu a^{\alpha\beta}\tau_\beta(\!theta\e, \!u\e, w\e)$
as new variables, and
write the model in a mixed form. The mixed model seeks $(\!theta\e, \!u\e, w\e)\in H$ and $(\!xi^{\eps}, \M^{\eps})\in V=[L^2]^5$
such that
\begin{multline}\label{N-P-model}
a(\!theta\e, \!u\e, w\e; \!phi, \!v, z)
+\int_{\t\Omega}
\left[\M^{\eps\alpha\beta}
\gamma_{\alpha\beta}(\!v,z)
+\xi^{\eps\alpha}\tau_\alpha(\!phi, \!v, z)\right]=
\langle \!f; \!phi, \!v, z\rangle\ \ \forall\
(\!phi, \!v,z) \in H,
\\
\int_{\t\Omega}\left[\N^{\alpha\beta}\gamma_{\alpha\beta}(\!u\e, w\e)+\eta^\alpha\tau_\alpha(\!theta\e, \!u\e, w\e)\right]
-
\eps^2\int_{\t\Omega}\left[a_{\alpha\beta\lambda\gamma}\M^{\eps\alpha\beta}\N^{\lambda\gamma}+\frac{1}{\kappa\mu}a_{\alpha\beta}\xi^{\eps\alpha}\eta^\beta\right]=0
\hfill
\\
\hfill\forall\ (\!eta, \N) \in V.\\
\end{multline}
Here $\!xi\e$ and $\!eta$ represent the vectors $\xi^{\eps\alpha}$ and $\eta^\alpha$, and
$\M\e$ and $\N$ denote the tensors $\M^{\eps\alpha\beta}$ and $\N^{\alpha\beta}$, respectively.
This mixed model  is the basis for the mixed finite element method.


\subsection{Asymptotic estimates on the shell model}
\label{abstract}
Notations in this sub-section are independent of the rest of the paper.
The results are presented in an abstract way, from which some useful estimates on the shell models \eqref{N-0-model} and \eqref{N-P-model} will follow.
For the shell model, there is no general theory on the detailed asymptotic structure of its solution. The results in this section provide some minimum
information that will be used in interpreting our results of numerical analysis and validating the procedure of
properly choosing a method to compute a given shell problem.

Let $H$,
$U$, and $V$ be Hilbert spaces, $A$ and $B$ be
linear continuous operators from $H$ to
$U$ and $V$, respectively.
We assume
\begin{equation}\label{equiva1as}
\|Av\|_U+\|Bv\|_V\simeq\|v\|_H\ \ \forall \ v\in H.
\end{equation}
For any $\eps>0$ and $f\in H^*$, the dual space of $H$, there is a unique $u\e\in H$, such that
\begin{equation}\label{prob1as}
(Au\e,Av)_U+\eps^{-2}(Bu\e,Bv)_V=\langle f,v\rangle
\quad \forall\ v\in H.
\end{equation}
The equivalence \eqref{equiva1as} also implies that we can introduce an equivalent norm by $\|v\|^2_{\H}:=\|Av\|^2_U+\|Bv\|^2_V$ for all $v\in H$. With such a norm
$\H$ is also a Hilbert space.
We let $\ker B\subset H$ be the kernel of the operator $B$, and
let $W\subset V$ be the range of $B$. We define a norm on $W$ by $\|w\|_{W}=\inf_{v\in H, Bv=w}\|v\|_H$ $\forall\ w\in W$, such that
$W$ is isomorphic to $(\ker B)^\perp_\H$.
We let $\overline W$ be the closure of $W$ in $V$.
Thus $W$ is a dense subset of $\overline W$, and  ${\overline W}^*$ is dense  in $W^*$.

The assumption \eqref{equiva1as}  also assures that the limiting problem
\begin{equation}\label{limitas}
(Au^0, Av)_U=\langle f, v\rangle\ \ \forall\ v\in\ker B
\end{equation}
has a unique  solution $u^0\in\ker B$. The functional  $\langle f,v\rangle-(Au^0, Av)_U$ annihilates $\ker B$. Therefore, according to the closed range theorem
there is a $\zeta\in W^*$ such that
\begin{equation*}
\langle f,v\rangle-(Au^0, Av)_U=\langle\zeta,Bv\rangle\ \ \forall\ v\in H.
\end{equation*}
Let $v\e\in H$ be the solution of
\begin{equation}\label{abs-membrane}
\eps^2(Av\e,Av)_U+(Bv\e,Bv)_V=\langle\zeta, Bv\rangle
\ \ \forall\ v\in H.
\end{equation}
We then have the expression that
\begin{equation}\label{abs-decomp}
u\e=u^0+\eps^2v\e.
\end{equation}
This is an orthogonal decomposition in $\H$, since $u^0\in\ker B$ and $v\e\in(\ker B)^\perp_{\H}$.

By
Theorem 2.1 of \cite{CR1},  we have
\begin{equation}\label{Thm2.1-1}
\eps\|Av\e\|_U+\|Bv\e\|_{V}\simeq \eps^{-1}\|\zeta\|_{W^*+\eps{\overline W}^*},
\end{equation}
\begin{equation}\label{Thm2.1-2}
\|\pi_{\overline W}Bv\e-\zeta\|_{W^*}+\eps\|Bv\e\|_V\simeq\|\zeta\|_{W^*+\eps{\overline W}^*}.
\end{equation}
Here $\pi_{\overline W}:\overline W\to{\overline W}^*$
is the inverse of Riesz representation.
If  $\zeta\in{\overline W}^*$,  by Theorem 2.2 of \cite{CR1}, we have the stronger estimate that
\begin{equation}\label{Thm2.2}
\eps\|v\e\|_H+\|Bv\e-i_{\overline W}\zeta\|_V+\eps^{-1}\|\pi_{\overline W}Bv\e-\zeta\|_{W^*}\simeq\|i_{\overline W}\zeta\|_{\eps W+\overline W}.
\end{equation}
Here $i_{\overline W}$ is the Riesz representation mapping from ${\overline W}^*$ to $\overline W$, which is the inverse of $\pi_{\overline W}$.
Since ${\overline W}^*$ is dense in $W^*$, we have $\lim_{\eps\to 0}\|\zeta\|_{W^*+\eps{\overline W}^*}=0$. If  $\zeta\in{\overline W}^*$
we have $\lim_{\eps\to 0}\|i_{\overline W}\zeta\|_{\eps W+\overline W}=0$ \cite{BL}.

The Naghdi shell model \eqref{N-model}, or equivalently \eqref{N-0-model}, fits in the form \eqref{prob1as} in an obvious manner.
The element $u\e$ represents the primary shell model solution $(\!theta\e, \!u\e, w\e)$.
The operator $A$ is the bending strain operator and $B$ combines the membrane and shear strain operators,
for which a proof of the condition \eqref{equiva1as} can be found in \cite{CiarletIII}.
The space $H$ is that defined by \eqref{N-space}. The spaces $U$ and $V$ are multiple $L^2$ spaces defined on $\Omega$.
The subspace $\ker B$ is composed of pure bending deformations.
The condition $\ker B\ne 0$ means the shell allows pure bending. The condition  $f|_{\ker B}\ne 0$ means that the loading  on the shell actually activates pure bending deformations.

A shell problem is bending dominated if and only $f|_{\ker B}\ne 0$.
In this case, in the expression \eqref{abs-decomp} we have $u_0\ne 0$.
For $w\in W\subset \overline W$, the function $\|\pi_{\overline W}w\|_{W^*}$ is a norm that is weaker than the norm of $\overline W$. We let $\overline{\overline W}$
be the completion of $W$ in this weaker norm.
In view of the estimates
\eqref{Thm2.1-1} and \eqref{Thm2.1-2}, we have
\begin{equation}\label{abs-bending-est}
\begin{gathered}
\lim_{\eps\to 0}u\e=u^0\ \text{ in } H\text{ and } u^0\ne 0, \\
\lim_{\eps\to 0}\eps^{-2}Bu\e=\xi^0\ \text{ in }\overline{\overline W}.
\end{gathered}
\end{equation}
The second convergence means that {\em
the scaled
membrane stress tensor and scaled shear stress vector  $(\M\e, \!xi\e)$ converges to a finite limit,
}
which was introduced in the mixed formulation \eqref{N-P-model}
as an independent variable and will be approximated independently in the finite element model.
This convergence is, however,  in an abstract norm that might not be explainable in terms of  Sobolev norms.
The quantity $\eps^2(Au\e, Au\e)_U$ represents the bending strain energy, and $(Bu\e, Bu\e)_V$ represents the sum of membrane and shear strain
energies.
We shall call $\eps^2(Au\e, Au\e)_U+(Bu\e, Bu\e)_V$ the total strain energy.
It is seen from \eqref{abs-decomp}, \eqref{Thm2.1-1} and \eqref{Thm2.1-2} that
${(Bu\e, Bu\e)_V}/{[\eps^2(Au\e, Au\e)_U]}\lesssim \|\zeta\|^2_{W^*+\eps{\overline W}^*}\to 0\ \ (\eps\to 0)$.
This explains the terminology of bending dominated shells.

A shell problem is not bending dominated if $f|_{\ker B}=0$.  This is to say that $u^0=0$ in \eqref{abs-decomp}, or formally the shell only shows
higher order response to the loading force.  In this case, we have $\langle f, v\rangle=\langle\zeta, Bv\rangle$ for all $v\in H$.
If $\zeta\in {\overline W}^*$, the shell is membrane/shear dominated.
Otherwise, it is intermediate.
For membrane/shear dominated shells, we see from \eqref{Thm2.2} that
\begin{equation}\label{membrane-goto-0}
\|u\e\|_H=\eps^2\|v\e\|_H\simeq \eps\|i_{\overline W}\zeta\|_{\eps W+\overline W}=o(\eps).
\end{equation}
This is strikingly different from the limiting behavior of bending dominated shells, see the first estimate in \eqref{abs-bending-est}.
For intermediate shells, we see from \eqref{Thm2.1-1} and \eqref{Thm2.1-2} that
\begin{equation}\label{intermediate-goto-0}
\|u\e\|_H\lesssim \|\zeta\|_{W^*+\eps{\overline W}^*}\to 0\ \ (\eps\to 0).
\end{equation}

For non-bending dominated shells, it is convenient
to consider the scaled deformation $v\e=\eps^{-2}u\e$. We use $u\e$, instead of $v\e$,  to denote the scaled deformation, which then is the solution of
 \begin{equation}\label{prob2as}
\eps^2(Au\e,Av)_U+(Bu\e,Bv)_V=\langle f,v\rangle
\quad \forall\ v\in H.
\end{equation}
Here $f$ is independent of $\eps$, and  there is a nonzero  $\zeta\in W^*$ such that $\langle f, v\rangle=\langle\zeta, Bv\rangle$ for all $v\in H$.
It is seen that $u\e\in (\ker B)^\perp_\H$ for all $\eps >0$.
The shell problem is membrane/shear dominated, if $f|_{\ker B}=0$ and $\zeta\in{\overline W}^*$. In this case $\|Bv\|_V$ defines a norm for  $v\in (\ker B)^\perp_\H$.
We use $\overline H$ to denote the completion of $(\ker B)^\perp_\H$ with respect to this norm.
From \eqref{Thm2.2} we see that $u\e$ converges in this norm. From \eqref{Thm2.2} we  also see that the total energy converges to a nonzero limit. We have
\begin{equation}\label{abs-membrane-est}
\begin{gathered}
\lim_{\eps\to 0}u\e=u^0\ \text{ in } \overline H,\\
\eps^2(Au\e,A u\e)_U+(Bu\e,Bu\e)_V\to \|i_{\overline W}\zeta\|^2_V\ne 0.
\end{gathered}
\end{equation}
We also see from \eqref{Thm2.2} that the problem is membrane/shear dominated since
\begin{equation*}
{\eps^2(Au\e, Au\e)_U}/{(Bu\e, Bu\e)_V}\lesssim \|i_{\overline W}\zeta\|_{\eps W+\overline W}
\to 0\ \ (\eps\to 0).
\end{equation*}

The shell problem is intermediate, if $f|_{\ker B}=0$ and $\zeta\not\in{\overline W}^*$. In this case $\|\pi_{\overline W}Bv\|_{W^*}$ defines a norm for  $v\in (\ker B)^\perp_\H$
that is weaker than the $\overline H$ norm. We denote the completion of $(\ker B)^\perp_\H$ with respect to this weaker norm by $\overline{\overline H}$.
From \eqref{Thm2.1-2} we see that $u\e$ converges to a limit $u^0$ in $\overline{\overline H}$. From \eqref{Thm2.1-1} we have an equivalence estimate
on the total energy. These are summarized as
\begin{equation}\label{abs-intermediate-est}
\begin{gathered}
\lim_{\eps\to 0}u\e=u^0 \ \text{ in } \overline{\overline H},\\
\eps^2(Au\e,A u\e)_U+(Bu\e,Bu\e)_V\simeq  \eps^{-2}\|\zeta\|^2_{W^*+\eps{\overline W}^*}.
\end{gathered}
\end{equation}
As $\eps\to 0$, the total energy tends to infinity, and it blows up to infinity at a rate of $o(\eps^{-2})$.

\subsection{An equivalent estimate for the mixed model}
\label{mixed-equivalence}
The mixed form of the Naghdi shell model \eqref{N-P-model} fits in the abstract problem
\eqref{isomorphism-abs} below. 
Let $H, V$ be Hilbert spaces. Let $a(\cdot,\cdot)$ and $c(\cdot,\cdot)$ be symmetric bilinear forms on $H$ and $V$, and $b(\cdot,\cdot)$
be a bilinear form on $H\x V$. We assume that there is  a constant $C$ such that
\begin{equation}\label{isomorphism-condition}
\begin{gathered}
|a(u, v)|\le C\|u\|_H\|v\|_H, \quad C^{-1}\|u\|_H^2\le a(u, u)\ \forall\ u, v\in H,\\
|c(p, q)|\le C\|p\|_V\|q\|_V, \quad C^{-1}\|p\|_V^2\le c(p, p)\ \forall\ p, q\in V,\\
|b(v, q)|\le C\|v\|_H\|q\|_V\ \forall\ v\in H, q\in V.
\end{gathered}
\end{equation}
For  $f\in H^*$ and $g\in V^*$,
we seek $u\in H$ and $p\in V$ such that
\begin{equation}\label{isomorphism-abs}
\begin{gathered}
a(u, v)+b(v, p)=\langle f, v\rangle\ \ \forall\ v\in H,\\
b(u, q)-\eps^2c(p, q)=\langle g, q\rangle\ \ \forall\ q\in V.
\end{gathered}
\end{equation}
This problem  has a unique solution in the space $H\x V$ \cite{ABrezzi2},
for which we need an accurate estimate.
We define a weaker (semi) norm on $V$ by
\begin{equation}\label{isomorphism-weak}
|q|_{\overline V}=\sup_{v\in H}\frac{b(v,q)}{\|v\|_H}\ \ \forall\ q\in V.
\end{equation}
We have the following
equivalence result, in which the equivalent constant is only dependent on the constant $C$ in \eqref{isomorphism-condition}.
\begin{multline}\label{isomorphism-equiv}
\|u\|_H+|p|_{\overline V}+\eps\|p\|_V\simeq
\sup_{v\in H, q\in V}\frac{a(u, v)+b(v, p)-
b(u, q)+\eps^2c(p, q)}{\|v\|_H+|q|_{\overline V}+\eps\|q\|_V}
\\
\forall\ u\in H, \ p\in V.
\end{multline}
A proof of this result can be obtained by using the techniques of proofs of Proposition 2.2 of \cite{Brezzi-Fortin-stable} and
Proposition 7.2.9 of \cite{Bathe-book}.
This equivalence will be used to analyze the finite element model, for which we verify the conditions \eqref{isomorphism-condition}
with the constant $C$ being independent of the finite element mesh and shell thickness.

\section{The finite element model}
\label{FEMmodel}
We assume that the coordinate domain $\Omega$ is a polygon.
On $\Omega$, we introduce a triangulation $\T_h$
that is shape regular but not necessarily quasi-uniform.
The shape regularity of a triangle is defined as the ratio of the diameter
of its smallest circumscribed circle and the diameter of its largest inscribed circle. The shape regularity of a triangulation is the maximum of
shape regularities of all its triangular elements. When we say a $\T_h$ is shape regular we mean that the shape regularity
of $\T_h$ is bounded by an absolute constant $\K$.  Shape regular meshes allow local refinements, and are capable of
resolving the boundary and internal layers that are very common in  solutions of the shell model.
We also use $\T_h$ to denote the set of all (open) triangles of the triangulation, and
let $\Omega_h=\cup_{\tau\in\T_h}\tau$.
We use $\E^0_h$ to denote both
the union of interior edges and the set of all interior edges.
The set of edges on the boundary $\partial\Omega$
is denoted by $\E^{\partial}_h$ that is divided as $\E^D_h\cup\E^S_h\cup \E^F_h$, corresponding to clamped, soft-simply supported, and free
portions of the shell boundary. We assume that points on the boundary where boundary condition changes type are included in the
finite element nodes.
We let
$\E_h=\E^0_h\cup\E^{\partial}_h$. For $\tau\in \T_h$, we use $h_\tau$ to denote its diameter, and for $e\in\E^h$, we use $h_e$ to denote its length.
The triangulation $\T_h$ is mapped to a curvilinear triangulation
$\t\T_h$ of the shell mid surface $\t\Omega$. We let $\t\Omega_h=\cup_{\t\tau\in\t\T_h}\t\tau$, and use $\tilde\E_h=\!Phi(\E_h)$ to denote the curvilinear
edges on the shell mid-surface $\t\Omega$ in a similar way. When we say a function is piecewise smooth we mean it is smooth on each $\tau\in\T_h$ or
on each $\t\tau\in\t\T_h$.

In analogue to the bilinear forms \eqref{bilinear-forms}, we define their discrete versions as follows.
For any discontinuous piecewise smooth vectors $\theta_\alpha$, $\phi_\alpha$,
$u_\alpha$, and $v_\alpha$, and
discontinuous piecewise scalars $w$ and $z$,
\begin{multline}\label{rho-h}
\rho_h(\!theta, \!u, w; \!phi,  \!v, z)=
\frac13\left\{\int_{\t\Omega_h}
a^{\alpha\beta\lambda\gamma}\rho_{\lambda\gamma}(\!theta, \!u, w)
\rho_{\alpha\beta}(\!phi, \!v, z)\right.\\
-\int_{\t\E^0_h}\left[a^{\alpha\beta\lambda\gamma}\lbrac\rho_{\lambda\gamma}(\!theta, \!u, w)\rbrac
\lbra\phi_\alpha\rbra_{n_\beta}
+a^{\alpha\beta\lambda\gamma}\lbrac\rho_{\lambda\gamma}(\!phi, \!v, z)\rbrac
\lbra\theta_\alpha\rbra_{n_\beta}\right]
\\
+\int_{\t\E^0_h}\left[a^{\alpha\beta\lambda\gamma}\lbrac\rho_{\lambda\gamma}(\!theta, \!u, w)\rbrac b^\delta_\alpha
\lbra v_\delta\rbra_{n_\beta}
+a^{\alpha\beta\lambda\gamma}\lbrac\rho_{\lambda\gamma}(\!phi, \!v, z)\rbrac b^\delta_\alpha
\lbra u_\delta\rbra_{n_\beta}\right]
\\
-\int_{\t\E^D_h}\left[a^{\alpha\beta\lambda\gamma}\rho_{\lambda\gamma}(\!theta, \!u, w)\phi_\alpha{n_\beta}
+a^{\alpha\beta\lambda\gamma}\rho_{\lambda\gamma}(\!phi, \!v, z)\theta_\alpha{n_\beta}\right]
\\
\left.+\int_{\t\E^{D\cup S}_h}\left[a^{\alpha\beta\lambda\gamma}\rho_{\lambda\gamma}(\!theta, \!u, w)b^\delta_\alpha
v_\delta{n_\beta}+a^{\alpha\beta\lambda\gamma}\rho_{\lambda\gamma}(\!phi, \!v, z)b^\delta_\alpha
u_\delta{n_\beta}\right]\right\}\\
+\C\sum_{e\in\E^0_h}
h^{-1}_e\int_{e}\left(\sum_{\alpha=1,2}\lbra \theta_{\alpha}\rbra \lbra \phi_{\alpha}\rbra\right)
+\C\sum_{e\in\E^{D}_h}
h^{-1}_e\int_{e}\left(\sum_{\alpha=1,2}\theta_{\alpha}\phi_{\alpha}\right)
\end{multline}

\begin{multline}\label{gamma-h}
\gamma_h(\!u, w; \!v, z)=
\int_{\t\Omega_h}a^{\alpha\beta\lambda\gamma}\gamma_{\lambda\gamma}(\!u, w)
\gamma_{\alpha\beta}(\!v,z)
\\
-\int_{\t\E^0_h}\left[
a^{\delta\beta\alpha\gamma}\lbrac\gamma_{\alpha\gamma}(\!u, w)\rbrac
\lbra v_\delta\rbra_{n_\beta}
+a^{\delta\beta\alpha\gamma}
\lbrac\gamma_{\alpha\gamma}(\!v, z)\rbrac\lbra u_\delta\rbra_{n_\beta}\right]\\
-\int_{\t\E^{D\cup S}_h}\left[
a^{\delta\beta\alpha\gamma}
\gamma_{\alpha\gamma}(\!u, w)
v_\delta{n_\beta}
+a^{\delta\beta\alpha\gamma}
\gamma_{\alpha\gamma}(\!v, z)
u_\delta{n_\beta}\right]\\
+\C\sum_{e\in\E^0_h}
h^{-1}_e\int_{e}\left(\sum_{\alpha=1,2}\lbra u_{\alpha}\rbra \lbra v_{\alpha}\rbra
+\lbra w\rbra \lbra z\rbra\right)
+\C\sum_{e\in\E^{D\cup S}_h}
h^{-1}_e\int_{e}\left(\sum_{\alpha=1,2}u_{\alpha}v_{\alpha}
+wz\right)
\end{multline}

\begin{multline}\label{tau-h}
\tau_h(\!theta, \!u, w; \!phi, \!v, z)=
\int_{\t\Omega_h}
\kappa\mu a^{\alpha\beta}\tau_\beta(\!theta, \!u, w)\tau_\alpha(\!phi, \!v, z)\\
-\int_{\t\E^0_h}\left[
\kappa\mu a^{\alpha\beta}\lbrac\tau_\beta(\!theta, \!u, w)\rbrac\lbra z\rbra_{n_\alpha}
+\kappa\mu a^{\alpha\beta}\lbrac\tau_\beta(\!phi, \!v, z)\rbrac\lbra w\rbra_{n_\alpha}\right]\\
-\int_{\t\E^{D\cup S}_h}\left[
\kappa\mu a^{\alpha\beta}\tau_\beta(\!theta, \!u, w)z{n_\alpha}
+\kappa\mu a^{\alpha\beta}\tau_\beta(\!phi, \!v, z)w{n_\alpha}\right]\\
+\C\sum_{e\in\E^0_h}
h^{-1}_e\int_{e}
\lbra w\rbra \lbra z\rbra
+\C\sum_{e\in\E^{D\cup S}_h}
h^{-1}_e\int_{e}wz
\end{multline}
In the above definitions, the integral on $\t\Omega_h$ is taken with respect to its area element.
An inter-element edge $\t e\in\t\E^0_h$ is shared by elements $\t\tau_1$ and $\t\tau_2$. Piecewise
functions may have different values on the two elements, and thus discontinuous on $\t e$. The notation
$\lbrac\rho_{\sigma\tau}(\!phi, \!v, z)\rbrac$ represents the average of values of  $\rho_{\sigma\tau}(\!phi, \!v, z)$
from the sides of $\t\tau_1$ and $\t\tau_2$.
Let $\!n_\delta=n_{\delta\alpha}\!a^\alpha$ be the unit outward normal to $\t e$ viewed as boundary of $\t\tau_\delta$.
We have $\!n_1+\!n_2=0$ and  $n_{1\alpha}+n_{2\alpha}=0$.
We use $\lbra u_{\alpha}\rbra_{n_{\lambda}}=(u_{\alpha})|_{\t\tau_1} n_{1\lambda}+(u_{\alpha})|_{\t\tau_2} n_{2\lambda}$ to denote the jump of $u_\alpha$ over the edge
$\t e$ with respect to $n_\lambda$,
and use $\lbra w\rbra_{n_{\lambda}}=w|_{\t\tau_1} n_{1\lambda}+w|_{\t\tau_2} n_{2\lambda}$ to denote the jump of $w$ over the edge
$\t e$ with respect to $n_\lambda$,
etc. Like covariant derivatives, the jump with respect to the normal vector of a lower order tensor is a higher order tensor.
In $\lbra u_{\alpha}\rbra_{n_{\lambda}}$, $\alpha$ and $\lambda$ are treated as the usual Greek subscripts in second order tensors.
On the boundary $\t\E^{D\cup S}_h$,  $\!n=n_\alpha\!a^\alpha$ is the unit outward in-surface normal to $\partial\t\Omega$.
The last lines in \eqref{rho-h}, \eqref{gamma-h}, and \eqref{tau-h} are the penalty terms in which
the jumps (without subscript) $\lbra\theta_\alpha\rbra$,  $\lbra u_\alpha\rbra$, and $\lbra w\rbra$  are the absolute values of the differences
in the values of $\theta_\alpha$, $u_\alpha$, and $w$ from the two sides of $e$, respectively.

Let $\tau\in\T_h$ be an element, which is mapped to $\t\tau=\!Phi(\tau)\in \t\T_h$.
Let $e$ be one of the edges of $\tau$. We denote the unit outward normal to $e$ by $\bar\!n=\bar n_\alpha\!e^\alpha$ with $\!e^\alpha$ being the basis vectors
of $\R^2$. The outward normal to $\t e=\!Phi(e)$ is $\!n=n_\alpha\!a^\alpha$.
In the above definitions,  the integral on  a curved edge $\t e\in\t\E_h$ is taken with respect to its arc length.
Let $x_\alpha(s)$ be the arc length parameterization
of $e$, then $\!Phi(x_\alpha(s))$ is a parameterization of $\t e$, but may not in terms of the arc length of $\t e$. The correct formula for a curved line integral is
$\int_{\t e}f=\int_ef(x_\alpha(s))\sqrt{\sum_{\gamma,\beta=1, 2}a_{\gamma\beta}\frac{dx_\gamma}{ds}\frac{dx_\beta}{ds}}ds$.
For a vector valued function $f^\alpha$, we have the simpler formula
\begin{equation*}
\int_{\t e}f^\alpha n_\alpha=\int_ef^\alpha \bar n_\alpha \sqrt a ds.
\end{equation*}
This formula shows how to write the line integrals on curved  edges in $\t\E_h$ in \eqref{form_a} and \eqref{form_b}
as line integrals on straight edges in $\E_h$. For example,
\begin{equation*}
\int_{\t e}\kappa\mu a^{\alpha\beta}\lbrac\tau_\beta(\!phi, \!v, z)\rbrac\lbra w\rbra_{n_\alpha}=
\int_{e}\kappa\mu a^{\alpha\beta}\lbrac\tau_\beta(\!phi, \!v, z)\rbrac\lbra w\rbra_{\bar n_\alpha}\sqrt a ds.
\end{equation*}
It is noted that $n_\alpha$ is usually not constantly valued on $\t e$ but $\bar n_\alpha$ is constant on the edge $e$, which is an important property in analysis of
the finite element models.
We need to repeatedly do integration by parts on the shell mid-surface, by using the Green's theorem on surfaces.
Let $\tau\subset\Omega$ be a subdomain, which is mapped to
the subregion $\tilde\tau\subset\t\Omega$ by $\!Phi$.
Let $\!n=n_{\alpha}\!a^{\alpha}$ 
be the unit outward normal to the boundary $\partial\tilde\tau=\!Phi(\partial\tau)$
which is tangent to the surface $\t\Omega$.
Let $\bar n_\alpha\!e^\alpha$ be the  unit outward normal vector to $\partial\tau$ in $\R^2$. 
The Green's theorem \cite{GZ} says that for a vector field $f^\alpha$,
\begin{equation}\label{Green}
\int_{\tilde\tau}f^{\alpha}|_{\alpha}=
\int_{\partial\tilde\tau}f^{\alpha}n_{\alpha}=
\int_{\partial\tau}f^{\alpha}\bar n_{\alpha}\sqrt a.
\end{equation}

We define the discrete version of the  bilinear form $a$ \eqref{bilinear-form-a} by
\begin{equation}\label{form_a}
a_h(\!theta, \!u, w; \!phi,  \!v, z)=\rho_h(\!theta, \!u, w; \!phi, \!v, z)+\gamma_h(\!u, w; \!v, z)+\tau_h(\!theta, \!u, w; \!phi, \!v, z).
\end{equation}

For continuous piecewise smooth symmetric tensor $\M$ and  vector $\!xi$, and discontinuous piecewise  vector $\!phi$,
discontinuous piecewise vector $\!v$, and discontinuous piecewise
scalar $z$, we define the bilinear form $b_h$ by
\begin{multline}\label{form_b}
b_h(\M,\!xi; \!phi, \!v,z)=
\int_{\t\Omega_h}\left[\M^{\alpha\beta}\gamma_{\alpha\beta}(\!v, z)+\xi^\alpha\tau_\alpha(\!phi, \!v, z)\right]\\
-\int_{\t\E^0_h}\left(\M^{\alpha\beta}\lbra v_{\alpha}\rbra_{n_{\beta}}+\xi^\alpha\lbra z\rbra_{n_\alpha}\right)
 -\int_{\t\E^{D\cup S}_h}\left(
\M^{\alpha\beta}{n_{\beta}}v_{\alpha}+\xi^\alpha n_\alpha z\right).
\end{multline}
For continuous piecewise tensors $\M$, $\N$ and continuous piecewise vectors $\!xi$, $\!eta$ we define the bilinear form $c_h$ by
\begin{multline}\label{form_c}
c_h(\M,\!xi; \N,\!eta)=\int_{\t\Omega_h}\left(a_{\alpha\beta\gamma\delta}\M^{\gamma\delta}\N^{\alpha\beta}
+\frac{1}{\kappa\mu}a_{\alpha\beta}\xi^\alpha\eta^\beta\right).\hfill
\end{multline}


The finite element model is defined on a space of piecewise polynomials.
On an element $\tau$, we let $\P^k(\tau)$ be the space
of polynomials of degree $k$. The finite element functions are mostly taken in $\P^1(\tau)$. For the displacements, we need to
enrich the space if $\tau$ has edges on the free boundary. Let the edges of $\tau$ be $e_i$. If $\tau$ has one edge ($e_1$) on $\E^F_h$, we need to add two
quadratic polynomials to $\P^1(\tau)$, which are  defined by
\begin{equation*}
p^e_1=\lambda_1p_1+1,\quad p^e_2=\lambda_1p_2+\lambda_2.
\end{equation*}
Here $\lambda_i$ are the barycentric coordinates with respect to the vertex opposite to the edge $e_i$,
and $p_1$ and $p_2$ are linear functions in $\P^1(\tau)$ determined by
\begin{equation}\label{P3*}
\int_{\t\tau}p^e_\alpha q=0\ \forall\ q\in \P^1(\tau).
\end{equation}
It is clear that $p^e_1$ and $p^e_2$ are in $\P^2(\tau)$ and they are orthogonal to $\P^1(\tau)$ in $L^2(\tau)$ with the inner product weighted by $\sqrt a$,
and are constant and linear on $e_1$, respectively. With these two functions added to $\P^1(\tau)$, we obtain a $5$ dimensional subspace of $\P^2(\tau)$, which
is denoted by $\P^e(\tau)$. A function in $\P^e(\tau)$ is uniquely determined by its $L^2$ projection in $\P^1(\tau)$ and its moments of degree $0$ and $1$
on the edge $e_1$, all weighted by $\sqrt a$.
If $\tau$ has two edges ($e_2$ and $e_3$) on $\E^F_h$. We need to add $4$ cubic polynomials to $\P^1(\tau)$ for the displacement
variables:
\begin{equation*}
p^v_1=\lambda_2\lambda_3p_1+\lambda_3,\quad p^v_2=\lambda_2\lambda_3p_2+\lambda^2_3,\quad
p^v_3=\lambda_2\lambda_3p_3+\lambda_2,\quad p^v_4=\lambda_2\lambda_3p_4+\lambda^2_2.
\end{equation*}
Here for $j\in\{1, 2, 3, 4\}$, $p_j$ is in $\P^1(\tau)$, and determined by the condition that $\int_{\t\tau}p^v_jq=0$ for all $q\in\P^1(\tau)$.
With $p^v_j$ added to $\P^1(\tau)$, we obtain a subspace of $\P^3(\tau)$, which is denoted by $\P^v(\tau)$.
A function in $\P^v(\tau)$ is uniquely determined by its $L^2$ projection in $\P^1(\tau)$ and its moments of degree $0$ and $1$
on the edges $e_2$ and $e_3$, all weighted by $\sqrt a$.

The finite element space is defined by
\begin{equation}\label{FE-space}
\begin{split}
\H_h=&\{(\!phi, \!v, z); \phi_\alpha \text{ is discontinuous on }\Omega\text{ and } \phi_\alpha\in \P^1(\tau)\ \forall\ \tau\in\T_h, \\
& v_\beta \text{ and } z\text{ are discontinuous and } v_\beta, z\in \P^1(\tau) \text { if }\partial\tau\cap\E^F_h=\emptyset,\\
& v_\beta, z\in \P^e(\tau) \text { if }\partial\tau\cap\E^F_h\text{ has one edge },\\
& v_\beta, z\in \P^v(\tau) \text { if }\partial\tau\cap\E^F_h\text{ has two edges}\},\\
\V_h=&\{(\N, \!eta);\  \N^{\alpha\beta} \text{ and }\eta^\gamma \text{ are continuous piecewise linear functions  on }\Omega
\}.\hfill
\end{split}
\end{equation}

The finite element model for the shell in the bending dominated state is an analogue of
\eqref{N-P-model}. It seeks $(\!theta, \!u, w)\in \H_h$ and $(\M, \!xi)\in \V_h$ such that
\begin{equation}\label{N-fem}
\begin{gathered}
a_h(\!theta, \!u, w;\  \!phi, \!v, z)+b_h(\M, \!xi;\  \!phi, \!v, z)=\langle\!f;\!phi, \!v, z\rangle\ \ \forall\ (\!phi, \!v, z)\in \H_h, \\
b_h(\N, \!eta;\  \!theta, \!u, w)-\eps^2c_h(\M, \!xi;\  \N,\!eta)=0\ \ \forall\ (\N, \!eta)\in \V_h.
\end{gathered}
\end{equation}

The solution of the Naghdi  model \eqref{N-P-model} satisfies the equation \eqref{N-fem} when the test functions $\!phi$, $\!v$, $z$, $\N$, and $\!eta$
are arbitrary piecewise smooth functions, not necessarily polynomials. This says that the finite element model is consistent
with the shell model. The verification of consistency is a matter of repeatedly applying the Green's theorem \eqref{Green}.
This equation is in the form of \eqref{isomorphism-abs}. We shall define the norms in $\H_h$ and $\V_h$ later,
in which we prove that the finite element model \eqref{N-fem} is well posed
if the penalty constant $\C$ in \eqref{form_a} is sufficiently large, by verifying the conditions \eqref{isomorphism-condition}.
This penalty constant $\C$ could be dependent on the shell geometry and the shape regularity $\K$ of
the triangulation $\T_h$. It is, otherwise, independent of the triangulation $\T_h$ and the shell thickness.
One may simply replace $\P^e(\tau)$  by the richer $\P^2(\tau)$, and replace $\P^v(\tau)$ by $\P^3(\tau)$.
This would slightly increase the complexity, but not affect the stability or accuracy of the finite element method.

For intermediate and membrane/shear dominated shells, the finite element model is an analogue of
the original model \eqref{N-0-model}.
We need to remove the enrichment of finite element functions for the displacement variables of elements
having edges on the free boundary. Such  enrichment's would not help but thwart the stability of the finite element model.
In the definition of $\H_h$, we replace $\P^e(\tau)$ and $\P^v(\tau)$ by $\P^1(\tau)$. Let $\ub\H_h$ be the reduced finite element space
of discontinuous piecewise linear functions.
The finite element model seeks
$(\!theta, \!u, w)\in \ub\H_h$ such that
\begin{multline}\label{N-0-fem}
a_h(\!theta, \!u, w;\  \!phi, \!v, z)+
\eps^{-2}\left[\gamma_h(\!u, w;\!v, z)+\tau_h(\!theta, \!u, w; \!phi, \!v, z)\right]\\
=\langle\!f;\!phi, \!v, z\rangle\ \ \forall\ (\!phi, \!v, z)\in \ub\H_h.
\end{multline}
The problem with this model is about its well-posedness, which
can not be guaranteed
by taking a sufficiently big penalty constant $\C$ in \eqref{rho-h}, \eqref{gamma-h}, and \eqref{tau-h}. We shall prove
the continuity and coerciveness of the bilinear form in \eqref{N-0-fem} with respect to the total energy norm in the finite element space $\ub\H_h$
when the mesh satisfies the condition \eqref{m-stability-condition} below.
This model is consistent with the shell model \eqref{N-0-model}, whose verification is similar to that of the model \eqref{N-fem}.

Note that there is no boundary condition enforced on functions in the spaces $\ub\H_h$, $\H_h$, or $\V_h$.
The displacement and rotation boundary conditions are enforced by boundary penalty in a consistent manner, which is Nitsche's method
\cite{A-DG}. This is necessary for the stability when the shell is bending dominated, while not necessary for membrane/shear dominated and intermediate shells.

In implementing the methods, we include a parameter $\vtheta$ in the definition \eqref{form_a} of the bilinear form $a_h$, and let
\begin{equation}\label{form_a_vtheta}
a^{\vtheta}_h(\!theta, \!u, w; \!phi,  \!v, z)=\rho_h(\!theta, \!u, w; \!phi, \!v, z)+\vtheta\gamma_h(\!u, w; \!v, z)+\vtheta\tau_h(\!theta, \!u, w; \!phi, \!v, z).
\end{equation}
We arrange the order of the finite element degree of freedoms as follows.

1. The degree of freedoms of the primary variables $\!theta, \!u, w$
in $\ub\H_h$ without the enrichment for $\!u$ and $w$ on elements having edges on $\E^F_h$.

2. The enriched degree of freedoms for
$\!u$ and $w$ in $\H_h\setminus \ub\H_h$.

3. The degree of freedoms of the auxiliary variables  $\M$ and $\!xi$ in $\V_h$.

A program for the method \eqref{N-fem} then is obtained by setting $\vtheta=1$ in \eqref{form_a_vtheta}.
A program for the method \eqref{N-0-fem} is obtained by taking $\vtheta=\veps^{-2}$, and
cutting out the left-upper sub-matrix of the stiffness matrix corresponding to the primary variables without the enrichment.

\section{A discrete Korn's inequality for Naghdi shell}
\label{KornOnShell}
To analyze the well posedness and stability of the finite element models \eqref{N-fem} and \eqref{N-0-fem},
we need to have a Korn type inequality.
Let $H^1_h$ be the space of piecewise $H^1$ functions in which a function is independently defined on each element $\tau$ of $\T_h$,
and belong to $H^1(\tau)$. For functions $\theta_\alpha$, $u_\alpha$, and $w$ in $H^1_h$, we defined the following norms and seminorms.
\begin{multline}\label{rho-h-norm}
\|(\!theta,\!u, w)\|^2_{\rho_h}:=
\sum_{\tau\in\T_h}\sum_{\alpha,\beta=1,2}\|\rho_{\alpha\beta}(\!theta, \!u, w)\|^2_{0,\tau}\\
+\sum_{e\in \E^0_h}h^{-1}_e
\int_{e}\sum_{\alpha=1,2}\lbra \theta_\alpha\rbra^2
+\sum_{e\in\E^D_h}h^{-1}_e\int_{e}\sum_{\alpha=1,2}\theta_\alpha^2.
\end{multline}

\begin{multline}\label{gamma-h-norm}
\|(\!u, w)\|^2_{\gamma_h}:=
\sum_{\tau\in\T_h}\sum_{\alpha,\beta=1,2}\|\gamma_{\alpha\beta}(\!u, w)\|^2_{0,\tau}
\\
+\sum_{e\in \E^0_h}h^{-1}_e
\int_{e}\sum_{\alpha=1,2}\lbra u_\alpha\rbra^2
+
\sum_{e\in\E^{S}_h\cup\E^D_h}
h^{-1}_e\int_{e}\sum_{\alpha=1,2}u^2_{\alpha}.
\end{multline}

\begin{multline}\label{tau-h-norm}
\|(\!theta,\!u, w)\|^2_{\tau_h}:=
\sum_{\tau\in\T_h}\sum_{\alpha=1,2}\|\tau_\alpha(\!theta, \!u, w)\|^2_{0,\tau}
+\sum_{e\in \E^0_h}h^{-1}_e\int_{e}\lbra w\rbra^2
+\sum_{e\in\E^{S}_h\cup\E^D_h}
h^{-1}_e\int_{e}w^2.\hfill
\end{multline}

\begin{multline}\label{ah-norm}
\|(\!theta,\!u, w)\|^2_{a_h}:=\|(\!theta,\!u, w)\|^2_{\rho_h}+\|(\!u, w)\|^2_{\gamma_h}+\|(\!theta,\!u, w)\|^2_{\tau_h}.\hfill
\end{multline}

\begin{multline}\label{Hh-norm}
\|(\!theta, \!u, w)\|^2_{\H_h}:=\sum_{\tau\in\T_h}\left[\sum_{\alpha=1,2}\left(\|\theta_\alpha\|^2_{1,\tau}+\|u_\alpha\|^2_{1,\tau}\right)+\|w\|^2_{1,\tau}\right]\\
+\sum_{e\in \E^0_h}h^{-1}_e
\int_{e}\left[\sum_{\alpha=1,2}\left(\lbra \theta_\alpha\rbra^2+\lbra u_\alpha\rbra^2\right)+
\lbra w\rbra^2\right]\\
+
\sum_{e\in\E^{S}_h\cup\E^D_h}
h^{-1}_e\int_{e}\left(\sum_{\alpha=1,2}u^2_{\alpha}+w^2\right)
+\sum_{e\in\E^D_h}
h^{-1}_e\int_{e}\sum_{\alpha=1,2}\theta_\alpha^2.
\end{multline}

We then have the equivalence that
\begin{equation}\label{Hh-ah-equiv}
\|(\!theta, \!u, w)\|_{a_h}\simeq\|(\!theta,\!u, w)\|_{\H_h}\ \forall\ \theta_\alpha, u_\beta, , w \in H^1_h.
\end{equation}
The equivalence constant $C$ could be dependent
on the shell midsurface and shape regularity $\K$ of the triangulation $\T_h$, but otherwise independent of the triangulation.
A  proof of this result can be found in \cite{Korn-shell}.

\section{Analysis of the mixed method for bending dominated shells}
\label{BendingErrorAnalysis}
The finite element model defined by \eqref{N-fem} fits in the form of the mixed equation
\eqref{isomorphism-abs}. We verify the conditions \eqref{isomorphism-condition} for the bilinear forms
defined by \eqref{form_a}, \eqref{form_b}, and \eqref{form_c}, with the space defined by \eqref{FE-space},
in which the $\H_h$ norm is defined by \eqref{Hh-norm}. We define the $\V_h$ norm by
\begin{equation}\label{Vh-norm}
\|(\N, \!eta)\|_{\V_h}:=\left(\sum_{\alpha,\beta=1,2}\|\N^{\alpha\beta}\|^2_{0,\Omega}+ \sum_{\alpha=1,2}\|\eta^\alpha\|^2_{0,\Omega}\right)^{1/2}.
\end{equation}
We show that there is a constant $C$ that depends on the shell geometry and shape regularity $\K$ of the triangulation $\T_h$, but otherwise, independent of the
triangulation such that
\begin{align} 
|a_h(\!theta, \!u, w; \!phi, \!v, z)| &\le C\|(\!theta, \!u, w)\|_{\H_h}\|(\!phi, \!v, z)\|_{\H_h}&\  &\forall\ (\!theta, \!u, w), (\!phi, \!v, z)\in \H_h,   \label{a-condition1}      \\
\|(\!phi, \!v, z)\|_{\H_h}^2& \le Ca_h(\!phi, \!v, z; \!phi, \!v, z)&\  &\forall\  (\!phi, \!v, z)\in \H_h,\label{a-condition2} \\
|b_h(\N, \!eta; \!phi, \!v, z)| &\le C\|(\!phi, \!v, z)\|_{\H_h}\|(\N, \!eta)\|_{\V_h}&\  &\forall\ (\!phi, \!v, z)\in \H_h, (\N, \!eta) \in \V_h,\label{b-condition}\\
|c_h(\M, \!xi;  \N, \!eta)|&\le C\|(\M, \!xi)\|_{\V_h}\|(\N, \!eta)\|_{\V_h}&\  &\forall\ (\M, \!xi), (\N, \!eta) \in \V_h,\label{c-condition1}\\
\|(\N, \!eta)\|_{\V_h}^2&\le Cc_h(\N, \!eta;  \N, \!eta)&\  &\forall\ (\N, \!eta) \in \V_h. \label{c-condition2}
\end{align} 
We start with \eqref{a-condition1}.
In expression \eqref{form_a} of the bilinear form $a_h$,
the integrals on $\t\Omega_h$ are easily bounded as follows.
\begin{multline*}
\int_{\t\Omega_h}\left|\left[a^{\alpha\beta\lambda\gamma}
\rho_{\lambda\gamma}(\!theta, \!u, w)\rho_{\alpha\beta}(\!phi, \!v, z)
+a^{\alpha\beta\lambda\gamma}
\gamma_{\lambda\gamma}(\!u, w)\gamma_{\alpha\beta}(\!v, z)
+a^{\alpha\beta}\tau_\alpha(\!theta, \!u, w)\tau_{\beta}(\!phi, \!v, z)\right]\right|\\
\lesssim\|(\!theta, \!u, w)\|_{a_h}\|(\!phi, \!v, z)\|_{a_h}
\lesssim\|(\!theta, \!u, w)\|_{\H_h}\|(\!phi, \!v, z)\|_{\H_h}.
\end{multline*}
In the expression \eqref{form_a}, there are many integrals on edges in $\t\E^0_h$, $\t\E^D$, and $\t\E^S$.
To estimate these, we need a trace theorem.
Let $\tau$ be a triangle, and $e$ one of its edges. Then there is a
constant $C$ depending on the shape regularity of $\tau$ such that
\begin{equation}\label{trace}
\int_eu^2\le C\left[h_e^{-1}\int_{\tau}u^2+\sum_{\alpha=1,2}h_e\int_{\tau}|\partial_\alpha u|^2\right]
\ \ \forall\ u\in H^1(\tau).
\end{equation}
This can be found in \cite{A-DG}.
Let $e\in\E^0_h$ be one of the interior edges shared by elements $\tau_1$ and $\tau_2$. We consider the estimate on a typical term.
Using 
the H\"older inequality, we get
\begin{multline*}
\left|\int_{\t e}
a^{\alpha\beta\lambda\gamma}\lbrac\rho_{\lambda\gamma}(\!phi, \!v, z)\rbrac\lbra\theta_\alpha\rbra_{n_\beta}\right|\lesssim
\left[\sum_{\lambda, \gamma=1,2}h_e\int_e\lbrac\rho_{\lambda\gamma}(\!phi, \!v, z)\rbrac^2
\right]^{1/2}
\left[h^{-1}_e\sum_{\alpha=1,2}\int_e\lbra \theta_\alpha\rbra^2
\right]^{1/2}.
\end{multline*}
Using the trace inequality \eqref{trace},  the formula \eqref{N-bending} and \eqref{covariant-derivative}, and
the inverse inequality for finite element functions, we have
\begin{multline*}
h_e\int_e\lbrac\rho_{\lambda\gamma}(\!phi, \!v, z)\rbrac^2
\lesssim
\sum_{\beta, \delta=1,2}\left(\int_{\tau_\delta}|\rho_{\lambda\gamma}(\!phi, \!v, z)|^2+
h^2_{\tau_\delta}\int_{\tau_\delta}|\partial_\beta\rho_{\lambda\gamma}(\!phi, \!v, z)|^2\right)\\
\lesssim
\sum_{\delta=1,2}\left[
\|\!phi\|^2_{1,\tau_\delta}+\|\!v\|^2_{1,\tau_\delta}+\|z\|^2_{0,\tau_\delta}+
h^2_{\tau_\delta}
\left(\|\!phi\|^2_{2,\tau_\delta}+\|\!v\|^2_{2,\tau_\delta}+\|z\|^2_{1,\tau_\delta}\right)
\right]
\\
\lesssim
\sum_{\delta=1,2}\left(
\|\!phi\|^2_{1,\tau_\delta}+\|\!v\|^2_{1,\tau_\delta}+\|z\|^2_{0,\tau_\delta}
\right).
\end{multline*}
Summing the above estimates for all $e\in\E^0_h$, and using Cauchy--Schwarz inequality, we get
\begin{multline}\label{a-5-est}
\left|\int_{\t\E^0_h}a^{\alpha\beta\lambda\gamma}\lbrac\rho_{\lambda\gamma}(\!phi, \!v, z)\rbrac\lbra\theta_\alpha\rbra_{n_\beta}\right|\\
\lesssim
\left[\sum_{\tau\in\T_h}\left(\|\!phi\|^2_{1,\tau}+\|\!v\|^2_{1,\tau}+\|z\|^2_{0,\tau}\right)\right]^{1/2}
\left(\sum_{e\in\E^0_h}h^{-1}_e\sum_{\alpha=1,2}\int_e\lbra \theta_\alpha\rbra^2\right)^{1/2}.
\end{multline}
All the other terms in the expression  \eqref{form_a} of the bilinear form of $a_h$  can be estimated similarly.
Estimates on the additional penalty terms in the expression \eqref{form_a}
can be obtained by using the Cauchy--Schwarz again.
This completes the proof of \eqref{a-condition1}.

Next, we consider \eqref{a-condition2}.
Let $B(\!phi, \!v, z; \!phi, \!v, z)$ be
the sum of all edge integrals in the expression \eqref{form_a} of $a_h(\!phi, \!v, z; \!phi, \!v, z)$, except the penalty terms involving the penalty
coefficient $\C$.
In view of the equivalence \eqref{Hh-ah-equiv}, there are constants $C_1$ and $C_3$ that depend on the shell mid-surface
and $\K$, and $C_2$ that depend on the penalty constant $\C$ in \eqref{form_a} such that
\begin{multline*}
a_h(\!phi, \!v, z; \!phi, \!v, z)\ge
C_1\|(\!phi, \!v, z)\|^2_{\H_h}\\
+C_2
\left[\sum_{e\in \E^0_h}\left(
\sum_{\alpha=1,2}h^{-1}_e\int_{e}\lbra v_\alpha\rbra^2+\sum_{\alpha=1, 2}
h^{-1}_e\int_{e}\lbra \phi_\alpha\rbra^2
+
h^{-1}_e\int_{e}\lbra z\rbra^2\right)\right.\\
\left.
+
\sum_{e\in\E^{S}_h\cup\E^D_h}
\left(\sum_{\alpha=1,2}h^{-1}_e\int_{e}v^2_{\alpha}+h^{-1}_e\int_{e}z^2\right)
+\sum_{e\in\E^D_h}
\sum_{\alpha=1,2}h^{-1}_e\int_{e}\phi_\alpha^2\right]
-
C_3|B(\!phi, \!v, z; \!phi, \!v, z)|
\end{multline*}
Using the same argument as what used in deriving \eqref{a-5-est}, we have an upper bound that
\begin{multline*}
|B(\!phi, \!v, z; \!phi, \!v, z)|\\
\lesssim
\|(\!phi, \!v, z)\|_{\H_h}
\left[\sum_{e\in \E^0_h}\left(\sum_{\alpha=1,2}
h^{-1}_e\int_{e}\lbra v_\alpha\rbra^2+\sum_{\alpha=1, 2}
h^{-1}_e\int_{e}\lbra \phi_\alpha\rbra^2
+
h^{-1}_e\int_{e}\lbra z\rbra^2\right)\right.\\
\left.
+
\sum_{e\in\E^{S}_h\cup\E^D_h}
\left(\sum_{\alpha=1,2}h^{-1}_e\int_{e}v^2_{\alpha}+h^{-1}_e\int_{e}z^2\right)
+\sum_{e\in\E^D_h}
\sum_{\alpha=1,2}h^{-1}_e\int_{e}\phi_\alpha^2
\right]^{1/2}.
\end{multline*}
It follows from Cauchy--Schwarz inequality that when the penalty constant $\C$ in \eqref{form_a} is sufficiently big (which makes
$C_2$ sufficiently big)
there is a $C$ such that \eqref{a-condition2} holds.

To see the the continuity \eqref{b-condition} of the bilinear form $b_h$, we only need to look at the second term
in the right hand side of \eqref{form_b}. For an $e\in\E^0_h$ shared by $\tau_1$ and $\tau_2$, we have
\begin{multline*}
\left|\int_e\lbrac\N^{\alpha\beta}\rbrac\lbra v_{\alpha}\rbra_{n_{\beta}}\right|
\lesssim
\left[\sum_{\alpha,\beta=1,2}h_e\int_e(\N^{\alpha\beta})^2
\right]^{1/2}
\left[\sum_{\alpha=1,2}h^{-1}_e\int_e\lbra v_{\alpha}\rbra^2
\right]^{1/2}
\\
\hfill \lesssim
\left[\sum_{\alpha,\beta,\delta=1,2}(|\N^{\alpha\beta}|^2_{0,\tau_\delta}+h^2_{\tau_\delta}|\N^{\alpha\beta}|^2_{1,\tau_\delta})
\right]^{1/2}
\left[\sum_{\alpha=1,2}h^{-1}_e\int_e\lbra v_{\alpha}\rbra^2
\right]^{1/2}.
\end{multline*}
Similarly,
\begin{multline*}
\left|\int_e\lbrac\eta^\alpha\rbrac\lbra z\rbra_{n_\alpha}\right|
\lesssim
\left[\sum_{\alpha,\delta=1,2}(|\eta^{\alpha}|^2_{0,\tau_\delta}+h^2_{\tau_\delta}|\eta^{\alpha}|^2_{1,\tau_\delta})
\right]^{1/2}
\left[h^{-1}_e\int_e\lbra z\rbra^2
\right]^{1/2}.\hfill
\end{multline*}
Summing up these estimates for all $e\in\E^0_h$, and using inverse inequality to the finite element functions $\N$ and $\!eta$, we
obtain
\begin{equation*}
\left|\int_{\t\E^0_h}\left(\lbrac\N^{\alpha\beta}\rbrac\lbra v_{\alpha}\rbra_{n_{\beta}}+\lbrac\eta^\alpha\rbrac\lbra z\rbra_{n_\alpha}\right)\right|
\lesssim
\|(\N, \!eta)\|_{\V_h}\|(\!phi, \!v, z)\|_{\H_h}.
\end{equation*}
The conditions \eqref{c-condition1} and \eqref{c-condition2} are trivial. 
In view of the theory of Section~\ref{mixed-equivalence}, when the penalty constant $\C$ in \eqref{rho-h}, \eqref{gamma-h}, and \eqref{tau-h}  is sufficiently big,
the finite element model \eqref{N-fem} has a unique solution
in the finite element space \eqref{FE-space}.

Corresponding to the weak norm \eqref{isomorphism-weak}, we define a weaker (semi) $\overline\V_h$ norm for finite element function in $\V_h$ by
\begin{equation}\label{isomorphism-weak-N}
|(\N, \!eta)|_{\overline\V_h}:=\sup_{(\!phi, \!v, z)\in \H_h}\frac{b(\N, \!eta; \!phi, \!v, z)}{\|(\!phi, \!v, z)\|_{\H_h}}\ \ \forall\ (\N, \!eta) \in \V_h.
\end{equation}
We are now in a situation in  which the equivalence estimate \eqref{isomorphism-equiv} is applicable, from which we get the equivalence
\begin{multline*}
\|(\!theta, \!u, w)\|_{\H_h}+|(\M, \!xi)|_{\overline\V_h}+\eps\|(\M, \!xi)\|_{\V_h}\\
\simeq
\sup_{(\!phi, \!v, z)\in\H_h, (\N, \!eta)\in\V_h}
\frac{
\left[
\begin{gathered}
a_h(\!theta, \!u, w; \!phi, \!v, z)+b_h(\M, \!xi; \!phi, \!v, z)\\
-b_h(\N, \!eta; \!theta, \!u, w)+\eps^2
c_h(\M, \!xi; \N, \!eta)
\end{gathered}
\right]
} {\|(\!phi, \!v, z)\|_{\H_h}+|(\N, \!eta)|_{\overline\V_h}+\eps\|(\N, \!eta)\|_{\V_h}}\\
 \ \forall\
(\!theta, \!u, w)\in\H_h, (\M, \!xi)\in\V_h.
\end{multline*}
Let $\!theta\e, \!u\e, w\e,\M\e, \!xi\e$ be the solution to the mixed formulation of the Naghdi  model \eqref{N-P-model}, let
$\!theta^h, \!u^h, w^h, \M^h, \!xi^h$ be the finite element solution to the finite element model \eqref{N-fem},
and let $\!theta^I, \!u^I, w^I, \M^I, \!xi^I$ be an interpolation to the Naghdi  model solution from the finite element space.
Since the finite element method \eqref{N-fem} and the Naghdi model \eqref{N-P-model}
are consistent, we have
\begin{multline}\label{error-fraction}
\|(\!theta^h-\!theta^I, \!u^h-\!u^I, w^h-w^I)\|_{\H_h}+\\
|(\M^h-\M^I, \!xi^h-\!xi^I)|_{\overline\V_h}+\eps\|(\M^h-\M^I, \!xi^h-\!xi^I)\|_{\V_h}
\simeq
\sup_{
(\!phi, \!v, z)\in\H_h,
(\N,\!eta)\in\V_h
}\\
\frac{\left[\begin{gathered}
a_h(\!theta\e-\!theta^I, \!u\e-\!u^I, w\e-w^I; \!phi, \!v, z)+b_h(\M\e-\M^I, \!xi\e-\!xi^I; \!phi, \!v, z)\\
-b_h(\N, \!eta; \!theta\e-\!theta^I, \!u\e-\!u^I, w\e-w^I)+\eps^2
c_h(\M\e-\M^I, \!xi\e-\!xi^I, \N, \!eta)\end{gathered}
\right]} {\|(\!phi, \!v, z)\|_{\H_h}+|(\N, \!eta)|_{\overline\V_h}+\eps\|(\N, \!eta)\|_{\V_h}}.
\end{multline}
We need to estimate the four terms in the numerator of last line above. For the first two terms we have the following estimates.
\begin{multline}\label{a-error}
\left|a_h(\!theta\e-\!theta^I, \!u\e-\!u^I, w\e-w^I;\!phi, \!v, z)\right|
\\
\lesssim
\left\{
\sum_{\tau\in\T_h}\left[
\sum_{k=0}^2h^{2k-2}_\tau\left( \sum_{\alpha=1,2}|\theta\e_\alpha-\theta^I_\alpha|^2_{k,\tau}+\sum_{\alpha=1,2}|u\e_\alpha-u^I_\alpha|^2_{k,\tau}
+|w\e-w^I|_{k,\tau}
\right)\right]
\right\}^{1/2}\\
\|(\!phi, \!v, z)\|_{\H_h}\ \forall\ (\!phi, \!v, z)\in \H_h.
\end{multline}
\begin{multline}\label{b2-error}
\left|b_h(\M\e-\M^I, \!xi\e-\!xi^I; \!phi, \!v,z)\right|\\
\lesssim
\left\{\sum_{\tau\in\T_h}
\left[\sum_{\alpha,\beta=1,2}\left(|\M^{\eps\alpha\beta}-\M^{I \alpha\beta}|^2_{0, \tau}+h^2_{\tau}|\M^{\eps\alpha\beta}-\M^{I\alpha\beta}|^2_{1, \tau}\right)\right.\right.\\
\left.\left.
+
\sum_{\alpha=1,2}\left(|\xi^{\eps\alpha}-\xi^{I \alpha}|^2_{0, \tau}+h^2_{\tau}|\xi^{\eps\alpha}-\xi^{I\alpha}|^2_{1, \tau}
\right)\right]
\right\}^{1/2}
\|(\!phi, \!v, z)\|_{\H_h}\ \forall\ (\!phi, \!v, z)\in\H_h.
\end{multline}
The proof of these estimates  is a matter of repeatedly using the trace theorem \eqref{trace} and Cauchy--Schwarz inequality.
In either  the inequality \eqref{a-error} or  \eqref{b2-error}, we did not
impose any condition for the interpolations $\!theta^I$, $\!u^I$, $w^I$, $\M^I$, and $\!xi^I$,  except that they
are finite element functions from the space \eqref{FE-space}. The next estimate is very different in that
the interpolation needs to be  particularly chosen to obtain a desirable bound for the third term in the numerator of the right hand side of \eqref{error-fraction}.
In view of \eqref{form_b}, we have
\begin{multline}\label{b-original}
b_h(\N, \!eta; \!theta\e-\!theta^I, \!u\e-\!u^I, w\e-w^I)\\
=
\int_{\t\Omega_h}\left[\N^{\alpha\beta}\gamma_{\alpha\beta}(\!u\e-\!u^I, w\e-w^I)+\eta^\alpha\tau_\alpha(\!theta\e-\!theta^I, \!u\e-\!u^I, w\e-w^I)\right]\\
-\int_{\t\E^0_h}\left(\lbrac\N^{\alpha\beta}\rbrac\lbra u\e_{\alpha}-u^I_\alpha\rbra_{n_{\beta}}+\lbrac\eta^\alpha\rbrac\lbra w\e-w^I\rbra_{n_\alpha}\right)
\\ -\int_{\t\E^{S}_h\cup\t\E^D_h}
\left[\N^{\alpha\beta}{n_{\beta}}(u\e_{\alpha}-u^I_\alpha)+\eta^\alpha n_\alpha (w\e-w^I)\right].
\end{multline}

For $\theta\e_\alpha$, on any $\tau\in\T_h$,
we define $\theta^I_\alpha\in \P^1(\tau)$ by the weighted $L^2(\tau)$ projection such that
\begin{equation}\label{thetaI}
\int_{\t\tau}(\theta\e_\alpha-\theta^I_\alpha)p=0\ \forall\ p\in \P^1(\tau).
\end{equation}

For $u\e_\alpha$ and $w\e$ , on a $\tau\in\T_h$, if $\partial\tau\cap\E^F_h=\emptyset$, we define $u^I_\alpha$ and $w^I$ in $\P^1(\tau)$ by
the weighted $L^2(\tau)$ projection such that
\begin{equation}\label{uwI-interior}
\int_{\t\tau}(u\e_\alpha-u^I_\alpha)p=0, \quad
\int_{\t\tau}(w\e-w^I)p=0\ \forall\ p\in P^1(\tau).
\end{equation}
If $\partial\tau\cap\E^F_h$ has one edge $e$, we define $u^I_\alpha$ and $w^I$ in $P^e(\tau)$ by
\begin{equation}\label{uwI-edge1}
\begin{gathered}
\int_{\t\tau}(u\e_\alpha-u^I_\alpha)p=0, \quad
\int_{\t\tau}(w\e-w^I)p=0\ \forall\ p\in \P^1(\tau), \\
\int_e(u\e_\alpha-u^I_\alpha) p\sqrt a=0, \quad
\int_e(w\e-w^I) p\sqrt a=0
\ \forall\ p\in \P^1(e).
\end{gathered}
\end{equation}
If $\partial\tau\cap\E^F_h$ has two edges $e_\beta$, we define $u^I_\alpha$ and $w^I$ in $P^v(\tau)$ by
\begin{equation}\label{uwI-edge2}
\begin{gathered}
\int_{\t\tau}(u\e_\alpha-u^I_\alpha)p=0, \quad \int_{\t\tau}(w\e-w^I)p=0
\ \forall\ p\in \P^2(\tau), \\
\int_{e_\beta}(u\e_\alpha-u^I_\alpha) p\sqrt a=0,\quad
\int_{e_\beta}(w\e-w^I) p\sqrt a=0
\ \forall\ p\in P^1(e_\beta).
\end{gathered}
\end{equation}
The unisolvences of \eqref{thetaI} and \eqref{uwI-interior} are trivial. The unisolvences of
\eqref{uwI-edge1} and \eqref{uwI-edge2}
are seen from the definition of the space $\P^e(\tau)$ and $\P^v(\tau)$.

\begin{lem}\label{b3-error-lem}
With the  interpolations defined by \eqref{thetaI} to \eqref{uwI-edge2},  we have
\begin{multline}\label{b3-error}
|b(\N, \!eta; \!theta\e-\!theta^I,\!u\e-\!u^I,w\e-w^I)|\\
\lesssim
\max_{\tau\in\T_h}h^2_\tau\left[\sum_{\alpha,\beta,\lambda=1,2}|\Gamma^{\lambda}_{\alpha\beta}|_{1,\infty,\tau}+
\sum_{\alpha,\beta=1,2}\left(|b_{\alpha\beta}|_{1,\infty,\tau}+|b^\beta_\alpha|_{1,\infty,\tau}\right)\right]\\
\|(\N, \!eta)\|_{\V_h}
\left[\sum_{\tau\in\T_h}h^{-2}_{\tau}
\left(\left|u\e_{\alpha}-u^I_{\alpha}\right|^2_{0,\tau}+
\left|w\e-w^I\right|^2_{0,\tau}\right)
\right]^{1/2}\ \ \forall\ (\N, \!eta)\in\V_h.
\end{multline}
\end{lem}
\begin{proof}
With an application of the Green's theorem \eqref{Green} on each element $\t\tau\in\t\T_h$, summing up, we obtain the following  alternative expression to \eqref{b-original}.
\begin{multline*}
b(\N, \!eta; \!theta\e-\!theta^I, \!u\e-\!u^I,w\e-w^I)
=
\int_{\t\Omega_h}\left[-\N^{\alpha\beta}|_{\beta}\left(u\e_{\alpha}-u^I_{\alpha}\right)-b_{\alpha\beta}\N^{\alpha\beta}
\left(w\e-w^I\right) \right.\\
\left. -\eta^\alpha|_\alpha(w\e-w^I)+\eta^\alpha(\theta\e_\alpha-\theta^I_\alpha)+\eta^\beta b^\alpha_\beta(u\e_\alpha-u^I_\alpha)
\right]\\
+\int_{\t\E_h^0}\lbra\N^{\alpha\beta}\rbra_{n_{\beta}}\lbrac u\e_{\alpha}-u^I_{\alpha}\rbrac+
\int_{\t\E_h^0}\lbra\eta^{\alpha}\rbra_{n_\alpha}\lbrac w\e-w^I\rbrac\\
+\int_{\t\E_h^F}\N^{\alpha\beta}n_{\beta}
\left(u\e_{\alpha}-u^I_{\alpha}\right)+\int_{\t\E_h^F}\eta^\alpha n_\alpha(w\e-w^I).
\end{multline*}
Since $\N^{\alpha\beta}$ and $\eta^\alpha$ are continuous piecewise linear polynomials, on each $e\in\E^0_h$ we have $\lbra\N^{\alpha\beta}\rbra_{n_{\beta}}=0$
and $\lbra\eta^{\alpha}\rbra_{n_\alpha}=0$.
For each $e\in\E^F_h$, we have, see \eqref{Green},
\begin{equation*}
\begin{gathered}
\int_{\t e}\N^{\alpha\beta}n_{\beta}
\left(u\e_{\alpha}-u^I_{\alpha}\right)=\int_{e}\N^{\alpha\beta}\bar n_{\beta}
\left(u\e_{\alpha}-u^I_{\alpha}\right)\sqrt a=0,\\
\int_{\t e}\eta^\alpha n_\alpha(w\e-w^I)=\int_e\eta^\alpha \bar n_\alpha(w\e-w^I)\sqrt a=0.
\end{gathered}
\end{equation*}
Using the formulas \eqref{covariant-derivative} for the covariant derivatives  $\N^{\alpha\beta}|_{\beta}$ and $\eta^\alpha|_\alpha$, and using properties
of the interpolations \eqref{thetaI} to \eqref{uwI-edge2}, the expression is further simplified to
\begin{multline}\label{b-simple}
b(\N, \!eta; \!theta\e-\!theta^I, \!u\e-\!u^I,w\e-w^I)
\\=
\int_{\t\Omega_h}\left[\left(b^\alpha_\beta\eta^\beta-
\Gamma^{\beta}_{\beta\gamma}\N^{\alpha\gamma}-
\Gamma^{\alpha}_{\delta\beta}\N^{\delta\beta}\right)
\left(u\e_{\alpha}-u^I_{\alpha}\right)-\left(\Gamma^\delta_{\delta\alpha}\eta^\alpha+b_{\alpha\beta}\N^{\alpha\beta}\right)
\left(w\e-w^I\right) \right].
\end{multline}
The last term is estimated as follows.
For $\tau\in\T_h$, we have
\begin{equation*}
\int_{\t\tau}b_{\alpha\beta}\N^{\alpha\beta}
\left(w\e-w^I\right)=
\int_{\t\tau}\left[b_{\alpha\beta}-p^0(b_{\alpha\beta})\right]\N^{\alpha\beta}
\left(w\e-w^I\right)
\end{equation*}
Here, $p^0(b_{\alpha\beta})$ is the best constant approximation to $b_{\alpha\beta}$
in the space $L^{\infty}(\tau)$ such that
\begin{equation*}
\left|b_{\alpha\beta}-p^0(b_{\alpha\beta})\right|_{0,\infty,\tau}\le Ch_\tau \left|b_{\alpha\beta}\right|_{1,\infty,\tau}.
\end{equation*}
From this, we see
\begin{equation*}
\left|\int_{\t\tau}b_{\alpha\beta}\N^{\alpha\beta}
\left(w\e-w^I\right)\right|
\lesssim
h^2_\tau\left|b_{\alpha\beta}\right|_{1,\infty,\tau}|\N^{\alpha\beta}|_{0, \tau}h^{-1}_\tau|w\e-w^I|_{0,\tau}.
\end{equation*}
Summing up such estimates for all $\tau\in\T_h$, and using Cauchy--Schwarz inequality, we get
\begin{multline*}
\left|\int_{\t\Omega_h}b_{\alpha\beta}\N^{\alpha\beta}
\left(w\e-w^I\right)\right|\\
\lesssim
\left[\max_{\tau\in\T_h}\left(h^2_\tau\sum_{\alpha,\beta=1,2}\left|b_{\alpha\beta}\right|_{1,\infty,\tau}\right)\right]\|\N\|_{0,\Omega_h}
\left(\sum_{\tau\in\T_h}h^{-2}_\tau|w\e-w^I|^2_{0,\tau}
\right)^{1/2}.
\end{multline*}
The other terms in \eqref{b-simple} can be estimated in the same way.
\end{proof}

It is trivial to see that
\begin{multline}\label{c-error}
\left|c(\M\e-\M^I, \!xi\e-\!xi^I; \N, \!eta)\right|\\
\lesssim \|(\N, \!eta)\|_{\V_h}
\left(\sum_{\alpha,\beta=1,2}|\M^{\eps\alpha\beta}-\M^{I\alpha\beta}|_{0,\Omega_h}
+\sum_{\alpha=1,2}|\xi^{\eps\alpha}-\xi^{I\alpha}|_{0,\Omega_h}
\right)\ \forall\ (\N, \!eta) \in\V_h.
\end{multline}

The next result  is obtained by  combining  \eqref{a-error}, \eqref{b2-error}, \eqref{b3-error}, \eqref{c-error}, and \eqref {error-fraction}.
\begin{lem}\label{N-fem-lem}
Let $(\!theta^h, \!u^h, w^h)$ and $(\M^h, \!xi^h)$ be the finite element solution determined by the finite element model
\eqref{N-fem}.
Let $\theta^I_\alpha$, $u^I_{\alpha}$, and $w^I$ be the interpolations to $\theta\e_\alpha$, $u\e_\alpha$, and $w\e$
in the finite element space \eqref{FE-space}, which is defined by the formulas \eqref{thetaI}, \eqref{uwI-interior},
\eqref{uwI-edge1}, and \eqref{uwI-edge2}, respectively. Let
$\M^{I\alpha\beta}$ and $\xi^{I\alpha}$ be approximations
to $\M^{\eps\alpha\beta}$ and $\xi^{\eps\alpha}$  to be selected from the space of continuous piecewise linear functions.  We have
\begin{multline}\label{N-fem-error-lem}
\|(\!theta^h-\!theta^I, \!u^h-\!u^I, w^h-w^I)\|_{\H_h}\\
+
|(\M^h-\M^I, \!xi^h-\!xi^I)|_{\overline\V_h}+\eps\|(\M^h-\M^I, \!xi^h-\!xi^I)\|_{\V_h}
\\
\lesssim
\left[1+\eps^{-1}
\max_{\tau\in\T_h}h^2_\tau\left(\sum_{\alpha,\beta,\lambda=1,2}|\Gamma^{\lambda}_{\alpha\beta}|_{1,\infty,\tau}+
\sum_{\alpha,\beta=1,2}|b_{\alpha\beta}|_{1,\infty,\tau}+\sum_{\alpha,\beta=1,2}|b^\beta_\alpha|_{1,\infty,\tau}\right)
\right]
\\
\left\{
\sum_{\tau\in\T_h}\left[
\sum_{k=0}^2h^{2k-2}_\tau\left(
\sum_{\alpha=1,2}
\left(|\theta\e_{\alpha}-\theta^I_{\alpha}|^2_{k, \tau}
+
|u\e_{\alpha}-u^I_{\alpha}|^2_{k, \tau}\right)
+
|w\e-w^I|^2_{k, \tau}\right)
\right.\right.\\
\left.\left.+\sum_{k=0}^1h^{2k}_\tau\left(\sum_{\alpha,\beta=1,2}|\M^{\eps\alpha\beta}-\M^{I\alpha\beta}|^2_{k, \tau}
+\sum_{\alpha=1,2}|\xi^{\eps\alpha}-\xi^{I\alpha}|^2_{k,\tau}\right)
\right]\right\}^{1/2}.
\end{multline}
\end{lem}
We have the following theorem on the error estimate for the finite element method \eqref{N-fem}.
\begin{thm}
There is a constant $C$ that only depends on the shell middle surface and the shape regularity $\K$ of the triangulation, but otherwise
is independent of the triangulation $\T_h$ and the shell thickness , such that
\begin{multline*}
\|(\!theta\e-\!theta^h, \!u\e-\!u^h, w\e-w^h)\|_{\H_h}\\
\le C
\left[1+\eps^{-1}
\max_{\tau\in\T_h}h^2_\tau\left(\sum_{\alpha,\beta,\lambda=1,2}|\Gamma^{\lambda}_{\alpha\beta}|_{1,\infty,\tau}+
\sum_{\alpha,\beta=1,2}|b_{\alpha\beta}|_{1,\infty,\tau}+\sum_{\alpha,\beta=1,2}|b^\beta_\alpha|_{1,\infty,\tau}\right)
\right]
\\
\left\{
\sum_{\tau\in\T_h}h^2_{\tau}\left[\sum_{\alpha=1,2}\left(\|\theta\e_\alpha\|^2_{2,\tau}+\|u\e_\alpha\|^2_{2,\tau}\right)+\|w\e\|^2_{2,\tau}\right]\right.\\
+
\left.\sum_{\tau\in\T_h}h^4_{\tau}\left[\sum_{\alpha, \beta=1,2}\|\M^{\eps\alpha\beta}\|^2_{2,\tau}+\sum_{\alpha=1,2}\|\xi^{\eps\alpha}\|^2_{2,\tau}
\right]\right\}^{1/2}.
\end{multline*}
Here $(\!theta^h, \!u^h, w^h)$ is the primary part of the solution of
the finite element model \eqref{N-fem} with the finite element space defined by \eqref{FE-space}. The norm $\|\cdot\|_{\H_h}$
is defined by \eqref{Hh-norm}.
\end{thm}
\begin{proof}
In view of the triangle inequality, we have
\begin{multline*}
\|(\!theta\e-\!theta^h, \!u\e-\!u^h, w\e-w^h)\|_{\H_h}\\
\le \|(\!theta\e-\!theta^I, \!u\e-\!u^I, w\e-w^I)\|_{\H_h}+\|(\!theta^h-\!theta^I, \!u^h-\!u^I, w^h-w^I)\|_{\H_h}.
\end{multline*}
Using the trace inequality \eqref{trace} to the edge terms in the norm  $\|(\!theta\e-\!theta^I, \!u\e-\!u^I, w\e-w^I)\|_{\H_h}$, cf., \eqref{Hh-norm}, we get
\begin{multline*}
\|(\!theta\e-\!theta^I, \!u\e-\!u^I, w\e-w^I)\|_{\H_h}\le C\\
\left\{
\sum_{\tau\in\T_h}\sum_{k=0, 1}h^{2k-2}_\tau\left[
\sum_{\alpha=1,2}\left(|u\e_{\alpha}-u^I_{\alpha}|^2_{k, \tau}+|\theta\e_{\alpha}-\theta^I_{\alpha}|^2_{k, \tau}\right)
+|w\e-w^I|^2_{k, \tau}
\right]\right\}^{1/2}.
\end{multline*}
For each $\tau\in\T_h$, we establish that
\begin{equation}\label{scaled-estimate}
\begin{gathered}
\sum_{k=0}^2h^{2k-2}_\tau|\theta\e_{\alpha}-\theta^I_{\alpha}|^2_{k, \tau}\le C h^2_\tau|\theta\e_\alpha|_{2,\tau},\\
\sum_{k=0}^2h^{2k-2}_\tau|u\e_{\alpha}-u^I_{\alpha}|^2_{k, \tau}\le C h^2_\tau|u\e_\alpha|_{2,\tau},\\
\sum_{k=0}^2h^{2k-2}_\tau|w\e-w^I|^2_{k, \tau}\le C h^2_\tau|w\e|^2_{2,\tau}.
\end{gathered}
\end{equation}

We scale $\tau$ to a similar triangle $\T$ whose diameter is $1$ by the scaling $X_\alpha=h^{-1}_\tau x_\alpha$. Let $\Theta_\beta(X_\alpha)=\theta\e_\beta(x_\alpha)$,
$U_\beta(X_\alpha)=u\e_\beta(x_\alpha)$,
$W(X_\alpha)=w\e(x_\alpha)$,
$A(X_\alpha)=a(x_\alpha)$, $\Theta^I_\beta(X_\alpha)=\theta^I_\beta(x_\alpha)$, $U^I_\beta(X_\alpha)=u^I_\beta(x_\alpha)$, and
$W^I(X_\alpha)=w^I(x_\alpha)$. It is easy to see that $\Theta^I_\alpha$ is the projection
of $\Theta_\alpha$ into $\P^1(\T)$ in the space $L^2(\T)$ weighted by $\sqrt{A(X_\alpha)}$.
This projection preserves linear  polynomials and we have the bound that
\begin{equation}\label{ThetaI-bound}
\|\Theta_\alpha^I\|_{0,\T}\le \left[\frac{\max_{\tau}a}{\min_{\tau}a}\right]^{1/4}\|\Theta_\alpha\|_{0,\T}.
\end{equation}
For a $\Theta_\alpha\in H^2(\T)$ and any linear polynomial $p$, using inverse inequality, there is a $C$ depending on the shape regularity of $\T$
such that
\begin{equation*}
\|\Theta_\alpha-\Theta^I_\alpha\|_{2, \T}\le \|\Theta_\alpha-p\|_{2, \T}+\|(\Theta_\alpha-p)^I\|_{2, \T}\le \|\Theta_\alpha-p\|_{2, \T}+C\|(\Theta_\alpha-p)^I\|_{0, \T}.
\end{equation*}
Therefore, there is a $C$ depending on the shape regularity of $\T$ and the  ratio ${\max_{\tau}a}/{\min_{\tau}a}$ such that
\begin{equation*}
\|\Theta_\alpha-\Theta^I_\alpha\|_{2, \T}\le C\|\Theta_\alpha-p\|_{2, \T}\ \ \forall\ p\in \P^1(\T).
\end{equation*}
Using the interpolation operator of \cite{Verfurth}, we can choose a $p\in \P^1(\T)$ and an absolute constant such that
 \begin{equation*}
\|\Theta_\alpha-\Theta^I_\alpha\|_{2, \T}\le C\|\Theta_\alpha-p\|_{2, \T}\le C|\Theta_\alpha|_{2, \T}
\end{equation*}
Scale this estimate from $\T$ to $\tau$, we obtain the first estimate in \eqref{scaled-estimate}.

If $\tau$ has no edge on the free boundary $\E^F_h$, the second inequality in \eqref{scaled-estimate} is proved
in the same way.
If $\tau$ has one or two edges on the free boundary, in place of the estimate
\eqref{ThetaI-bound}, we have that there is a $C$ depending only on the shape regularity of $\T$ such that
\begin{equation*} 
\|U^I_\alpha\|_{0,\T}\le C \|U_\alpha\|_{1,\T}\ \ \forall\ U_\alpha \in H^1(\T).
\end{equation*}
For any $p\in \P^1(\T)$, we have
\begin{multline*}
\|U_\alpha-U^I_\alpha\|_{2, \T}\le \|U_\alpha-p\|_{2, \T}+\|(U_\alpha-p)^I\|_{2, \T}\\
\le \|U_\alpha-p\|_{2, \T}+C\|(U_\alpha-p)^I\|_{0, \T}
\le \|U_\alpha-p\|_{2, \T}+C\|U_\alpha-p\|_{1, \T}
\le C\|U_\alpha-p\|_{2, \T}.
\end{multline*}
Using the interpolation operator of \cite{Verfurth} again, we get
\begin{equation*}
\|U_\alpha-U^I_\alpha\|_{2, \T}\le C|U_\alpha|_{2,\T}.
\end{equation*}
Here $C$ only depends on the shape regularity of $\T$.
The second inequality in \eqref{scaled-estimate} then follows the scaling from $\T$ to $\tau$.
The third one is the same as the second one with  $\alpha=1$ or $2$.

Finally, we need to show that there exist interpolations $\M^{I\alpha\beta}$ and $\xi^{I\alpha}$ from continuous piecewise linear functions for $\M^{\eps\alpha\beta}$
and $\xi^{\eps\alpha}$, respectively,
such that
\begin{equation*}
\begin{gathered}
\sum_{\tau\in\T_h}\left(|\M^{\eps\alpha\beta}-\M^{I\alpha\beta}|^2_{0, \tau}+h^2_{\tau}|\M^{\eps\alpha\beta}-\M^{I\alpha\beta}|^2_{1, \tau}\right)
\le C
\sum_{\tau\in\T_h}h^4_{\tau}\|\M^{\eps\alpha\beta}\|^2_{2,\tau},\\
\sum_{\tau\in\T_h}\left(|\xi^{\eps\alpha}-\xi^{I\alpha}|^2_{0, \tau}+h^2_{\tau}|\xi^{\eps\alpha}-\xi^{I\alpha}|^2_{1, \tau}\right)
\le C
\sum_{\tau\in\T_h}h^4_{\tau}\|\xi^{\eps\alpha}\|^2_{2,\tau}.
\end{gathered}
\end{equation*}
This requirement can be
met by choosing $\M^{I\alpha\beta}$ and $\xi^{I\alpha}$ as the nodal point interpolations  of $\M^{\eps\alpha\beta}$
and $\xi^{\eps\alpha}$, respectively.
In view of Lemma~\ref{N-fem-lem}, the proof is completed.
\end{proof}
This theorem shows that formally the primary part of the solution of the finite element model \eqref{N-fem}
approximates the Naghdi model solution. Indeed, when the shell deformation is bending dominated,
this estimate implies that the approximation is accurate. The reason is as follows.
First, according to \eqref{abs-bending-est}, we have  that  as $\eps\to 0$,
$\|(\!theta\e, \!u\e, w\e)\|_{\!H^1\x\!H^!\x H^1}$ converges to a nonzero number. Therefore, if the error is small, then
the relative error is small in the $\H^1_h$ norm. Second, since  $(\!theta\e, \!u\e, w\e)$ converges to a limiting function
in $H^1$, and $(\M\e, \!xi\e)$ converges to a finite limit in a weak norm, it is possible to make the
upper bound in the theorem small, as long as the finite element mesh is able to resolve the layers in these functions.
However, this result does not imply accuracy of the finite element solution if the shell problem is  membrane/shear dominated.
Using the asymptotic estimate \eqref{membrane-goto-0}, for membrane/shear dominated shells,  we have $\|(\!theta\e, \!u\e, w\e)\|=o(\eps)$,  but
$(\M\e, \!xi\e)$ still converges to a finite limit, in the weak norm,
and there is not relative smallness in the $\H^1_h$ norm of the error.
In view of \eqref{intermediate-goto-0}, the situation is similar for intermediate shells, and the finite element solution can not
be an accurate approximation to the shell model solution.

\section{The DG method for  membrane/shear and intermediate shells}
\label{MembraneErrorAnalysis}
As in the abstract setting described in  the paragraph of equation \eqref{prob2as},  for shell deformations that are not bending dominated,
to deal with the diminishing behavior of the model solution, it is convenient to assume the loading force components  in the Naghdi model are
$\eps^{-2}p^i$, $\eps^{-2}q^i$ and $\eps^{-2}r^\alpha$, with $p^i, q^i, r^\alpha$ being independent of $\eps$.
Under this assumption, the Naghdi model \eqref{N-0-model} seeks
$(\!theta\e, \!u\e, w\e)\in H$ such that
\begin{multline}\label{N-0-model-m}
\eps^2a(\!theta\e, \!u\e, w\e;\  \!phi, \!v, z)+
\gamma(\!u\e, w\e;\!v, z)+\tau(\!theta\e, \!u\e, w\e; \!phi, \!v, z)\\
=\langle\!f;\!phi, \!v, z\rangle\ \ \forall\ (\!phi, \!v, z)\in H.
\end{multline}
The finite element model
\eqref{N-0-fem} seeks
$(\!theta, \!u, w)\in \ub\H_h$ such that
\begin{multline}\label{N-0-fem-m}
\eps^2a_h(\!theta, \!u, w;\  \!phi, \!v, z)+
\gamma_h(\!u, w;\!v, z)+\tau_h(\!theta, \!u, w; \!phi, \!v, z)\\
=\langle\!f;\!phi, \!v, z\rangle\ \ \forall\ (\!phi, \!v, z)\in \ub\H_h.
\end{multline}
Here $\!f$ is still expressed by \eqref{form_f}, in which the loading functions $p^i$, $q^i$, and $r^\alpha$ are independent of $\eps$.
A major issue here is  whether it is true that the model \eqref{N-0-fem-m} can be made well posed
by taking the penalty constant $\C$ in the definition of \eqref{form_a}, \eqref{rho-h}, \eqref{gamma-h}, and \eqref{tau-h}
sufficiently big. Just like in the mixed method for bending dominated shells, there is no absolute stability for the discontinuous Galerkin method
for membrane/shear dominated and intermediate shells,  which holds uniformly with respect to $\eps$
and $\T_h$. Instead, we have the following result.
\begin{lem}
\label{m-stability-lem}
We assume that the triangulation satisfies the condition that there is a constant $C$ independent of $\T_h$ and $\eps$ such that
\begin{equation}\label{m-stability-condition}
\max_{\tau\in\T_h}h^2_\tau\left(\sum_{\alpha,\beta,\lambda,\delta=1,2}|\Gamma^{\lambda}_{\alpha\beta}|_{\delta,\infty,\tau}+
\sum_{\alpha,\beta,\delta=1,2}|b_{\alpha\beta}|_{\delta,\infty,\tau}+\sum_{\alpha,\beta,\delta=1,2}|b^\beta_\alpha|_{\delta,\infty,\tau}\right)
\le C\eps.
\end{equation}
Then when the penalty constant $\C$ (that is independent of $\T_h$ and $\eps$) in \eqref{rho-h}, \eqref{gamma-h}, and \eqref{tau-h} is big enough,
the  bilinear form
in the finite element model \eqref{N-0-fem-m} is continuous and coercive in the finite element space $\ub\H_h$
with respect to the discrete energy norm. I.e.,
There is a constant $C$ independent of $\T_h$ and $\eps$ such that for any
$(\!theta, \!u, w)$ and $(\!phi, \!v, z)$ in $\ub\H_h$
\begin{multline}\label{m-DG-continuous}
\eps^2a_h(\!theta, \!u, w;\  \!phi, \!v, z)+
\gamma_h(\!u, w;\!v, z)+\tau_h(\!theta, \!u, w; \!phi, \!v, z)\\
\le C\left[\eps\|(\!theta, \!u, w)\|_{a_h}+\|(\!u, w)\|_{\gamma_h}+\|(\!theta, \!u, w)\|_{\tau_h}\right]
\\
\left[\eps\|( \!phi, \!v, z)\|_{a_h}+\|(\!v, z)\|_{\gamma_h}+\|( \!phi, \!v, z)\|_{\tau_h}\right],
\end{multline}
\begin{multline}\label{m-DG-coercive}
\left[\eps\|(\!theta, \!u, w)\|_{a_h}+\|(\!u, w)\|_{\gamma_h}+\|(\!theta, \!u, w)\|_{\tau_h}\right]^2\\
\le C
\left[\eps^2a_h(\!theta, \!u, w;\  \!theta, \!u, w)+
\gamma_h(\!u, w; \!u, w)+\tau_h(\!theta, \!u, w; \!theta, \!u, w)\right].
\end{multline}
\end{lem}
\begin{proof}
In view of the definitions of the bilinear forms \eqref{gamma-h} and \eqref{tau-h}, and the definitions of the discrete norms and seminorms
\eqref{ah-norm}, \eqref{gamma-h-norm} and
\eqref{tau-h-norm}, we need to estimate a number of  integrals on the edges in $\E^0_h$, $\E^D_h$, and $\E^S_h$.
Let $e\in \E^0_h$ be an interior edge shared by $\tau_1$ and $\tau_2$.  We consider a typical term
\begin{equation}
\int_{\t e}
\lbrac\gamma_{\alpha\beta}(\!u, w)\rbrac
\lbra v_\delta\rbra_{n_\gamma}=\int_{e}
\lbrac\gamma_{\alpha\beta}(\!u, w)\rbrac
\lbra v_\delta\rbra_{\bar n_\gamma}.
\end{equation}
In the space $L^2(\tau_\sigma)$, we project $\gamma_{\alpha\gamma}(\!u, w)$ into $\P^1(\tau_\sigma)$,
and piece the projections
together to get a $p$ on $\tau_1\cup\tau_2$, which could be discontinuous over $e$.  Then
\begin{equation*}
\lbrac\gamma_{\alpha\beta}(\!u, w)\rbrac=\lbrac\gamma_{\alpha\beta}(\!u, w)-p\rbrac+\lbrac p\rbrac.
\end{equation*}
Using H\"older  inequality, the trace theorem \eqref{trace}, and the inverse inequality for linear functions, we have
\begin{multline}\label{p-est}
\left|\int_{e}
\lbrac p\rbrac
\lbra v_\delta\rbra_{\bar n_\gamma}\right|
\lesssim
\left[\sum_{\sigma=1,2}\left(
|p|^2_{0,\tau_\sigma}+h^2_{\tau_\sigma}|p|^2_{1,\tau_\sigma}\right)\right]^{1/2}h^{-1/2}_e|\lbra v_\delta\rbra|_{0,e}\\
\lesssim
\left[\sum_{\sigma=1,2}|\gamma_{\alpha\beta}(\!u, w)|_{0,\tau_\sigma}\right]h^{-1/2}_e|\lbra v_\delta\rbra|_{0,e}.
\end{multline}

Using H\"older inequality, the trace inequality \eqref{trace}, and error estimate of the $L^2$ projection, we get the estimate
\begin{multline}\label{gamma-p-est}
\left|\int_{\t e}
\lbrac\gamma_{\alpha\beta}(\!u, w)-p\rbrac
\lbra v_\delta\rbra_{\bar n_\gamma}\right|\lesssim \\
\left[\sum_{\sigma=1,2}\left(|\gamma_{\alpha\beta}(\!u, w)-p|^2_{0,\tau_\sigma}+h^2_{\tau_\sigma}|\gamma_{\alpha\beta}(\!u, w)-p|^2_{1,\tau_\sigma}\right)\right]^{1/2}
h^{-1/2}_e|\lbra v_\delta\rbra|_{0,e}\\
\le
\left[\sum_{\sigma=1,2}\left(h^4_{\tau_\sigma}|\gamma_{\alpha\beta}(\!u, w)|^2_{2,\tau_\sigma}\right)\right]^{1/2}
h^{-1/2}_e|\lbra v_\delta\rbra|_{0,e}.
\end{multline}
In view of the definition of the membrane strain operator \eqref{N-metric}, and using the fact that $u_\alpha$ and $w$ are piecewise linear functions,
we have
\begin{equation}\label{gamma-2-est}
|\gamma_{\alpha\beta}(\!u, w)|^2_{2,\tau}\le \left(|\Gamma^{\lambda}_{\alpha\beta}|^2_{2,\infty,\tau}+|\Gamma^{\lambda}_{\alpha\beta}|^2_{1,\infty,\tau}\right)
\|u_{\lambda}\|^2_{1,\tau}
+\left(|b_{\alpha\beta}|^2_{2, \infty, \tau}+|b_{\alpha\beta}|^2_{1,\infty, \tau}\right)\|w\|^2_{1,\tau}.
\end{equation}
Under the condition \eqref{m-stability-condition}, we have
\begin{multline}\label{gamma-est}
\left|\int_{\t e}
\lbrac\gamma_{\alpha\beta}(\!u, w)\rbrac
\lbra v_\delta\rbra_{\bar n_\gamma}\right|\lesssim h^{-1/2}_e|\lbra v_\delta\rbra|_{0,e}\\
\left[\sum_{\sigma=1,2}
\left(h^4_{\tau_\sigma}
 \sum_{\rho=1,2}\left(|\Gamma^{\lambda}_{\alpha\beta}|^2_{\rho,\infty,\tau_\sigma}\|u_\lambda\|^2_{1,\tau_\sigma}+|b_{\alpha\beta}|^2_{\rho, \infty, \tau_\sigma}
\|w\|^2_{1,\tau_\sigma}\right)+
|\gamma_{\alpha\beta}(\!u, w)|^2_{0,\tau_\sigma}\right)\right]^{1/2}\\
\lesssim
 h^{-1/2}_e|\lbra v_\delta\rbra|_{0,e}
\left[\sum_{\sigma=1,2}
\left(\eps^2\|\!u\|^2_{1,\tau_\sigma}+\eps^2\|w\|^2_{1,\tau_\sigma}+
|\gamma_{\alpha\beta}(\!u, w)|^2_{0,\tau_\sigma}\right)\right]^{1/2}.
\end{multline}
Similarly, we have
\begin{multline}\label{tau-est}
\left|\int_{\t e}
\lbrac\tau_{\alpha}(\!theta, \!u, w)\rbrac
\lbra z\rbra_{\bar n_\beta}\right|
\lesssim
 h^{-1/2}_e|\lbra z\rbra|_{0,e}
\left[\sum_{\sigma=1,2}\left(
\eps^2\|\!u\|^2_{1,\tau_\sigma}+
|\tau_{\alpha}(\!theta, \!u, w)|^2_{0,\tau_\sigma}\right)\right]^{1/2}.\hfill
\end{multline}

If the edge $e$ is in $\E^{D\cup S}_h$, then it is the edge of one element. We have similar estimates for
$\int_{\t e}
\gamma_{\alpha\beta}(\!u, w)
v_\delta n_\gamma$
and
$\int_{\t e}
\tau_{\alpha}(\!theta, \!u, w)
z n_\beta$,
except that $|\lbra v_\delta\rbra|_{0,e}$ and $|\lbra z\rbra|_{0,e}$ are replaced by $|v_\delta|_{0,e}$ and $|z|_{0,e}$, respectively.
It follows from the Korn's inequality \eqref{Hh-ah-equiv} that
\begin{equation}\label{Korn-last}
\eps^2\sum_{\tau\in\T_h}\left(\|\!u\|^2_{1,\tau}+\|w\|^2_{1,\tau}\right)\lesssim\eps^2\|(\!theta, \!u, w)\|^2_{a_h}.
\end{equation}
The continuity \eqref{m-DG-continuous} then follows from \eqref{a-condition1}, \eqref{a-condition2}, the inequality
\eqref{Korn-last}, and the above estimates \eqref{gamma-est} and \eqref{tau-est} on the edge integrals.

From the estimates \eqref{gamma-est} and \eqref{tau-est},
using the Cauchy--Schwarz inequality, we see that there is an absolute constant $\ub C$ such that for any $C\ge\ub C$
\begin{multline}\label{gamma-est-Cauchy}
\left|\int_{\t e}
\lbrac\gamma_{\alpha\beta}(\!u, w)\rbrac
\lbra u_\delta\rbra_{\bar n_\gamma}\right|
\le
C h^{-1}_e|\lbra u_\delta\rbra|^2_{0,e}
+\frac1C\sum_{\sigma=1,2}
\left(\eps^2\|\!u\|^2_{1,\tau_\sigma}+\eps^2\|w\|^2_{1,\tau_\sigma}+
|\gamma_{\alpha\beta}(\!u, w)|^2_{0,\tau_\sigma}\right)
\end{multline}
and
\begin{multline}\label{tau-est-Cauchy}
\left|\int_{\t e}
\lbrac\tau_{\alpha}(\!theta, \!u, w)\rbrac
\lbra w\rbra_{\bar n_\beta}\right|
\le
Ch^{-1}_e|\lbra w\rbra|^2_{0,e}
+\frac1C\sum_{\sigma=1,2}\left(
\eps^2\|\!u\|^2_{1,\tau_\sigma}+
|\tau_{\alpha}(\!theta, \!u, w)|^2_{0,\tau_\sigma}\right).\hfill
\end{multline}
Similar estimates hold for edge integrals on $\E^{D\cup S}_h$.
From these estimates and  the inequality \eqref{Korn-last},  it is  clear
to see that when the penalty constant $\C$ in \eqref{rho-h}, \eqref{gamma-h}, and \eqref{tau-h}
is big enough, we have the coerciveness \eqref{m-DG-coercive}.
\end{proof}

\begin{remark}
If the finite element method \eqref{N-0-fem-m} is defined in the boundary enriched finite element space $\H_h$, instead of $\ub\H_h$, then
a requirement stronger than \eqref{m-stability-condition} needs to be imposed on the finite element mesh
such that the continuity and coerciveness of Lemma~\ref{m-stability-lem} could hold.

Let $\tau$ be an element that has one or two edges  in $\E^F_h$. For finite element functions in the enriched space
$\H_h$,  the second derivatives of $u_\lambda$ and $w$ will not be zero on $\tau$, which need to be bounded  by
their first derivatives by using the inverse inequality.
In lieu of \eqref{gamma-2-est}, we now only have
\begin{multline}
|\gamma_{\alpha\beta}(\!u, w)|^2_{2,\tau}\le \left(|\Gamma^{\lambda}_{\alpha\beta}|^2_{2,\infty,\tau}+|\Gamma^{\lambda}_{\alpha\beta}|^2_{1,\infty,\tau}
+h^{-2}_\tau|\Gamma^{\lambda}_{\alpha\beta}|^2_{0,\infty,\tau}
\right)
\|u_{\lambda}\|^2_{1,\tau}\\
+\left(|b_{\alpha\beta}|^2_{2, \infty, \tau}+|b_{\alpha\beta}|^2_{1,\infty, \tau}+h^{-2}_\tau|b_{\alpha\beta}|^2_{0,\infty,\tau}\right)\|w\|^2_{1,\tau}.
\end{multline}
As a consequence, in addition to the condition \eqref{m-stability-condition}, we need to require that
\begin{equation*}
\max_{\overline \tau\cap\E^F_h\ne\emptyset}h_\tau\left(\sum_{\alpha,\beta,\lambda=1,2}|\Gamma^{\lambda}_{\alpha\beta}|_{0,\infty,\tau}+
\sum_{\alpha,\beta=1,2}|b_{\alpha\beta}|_{0,\infty,\tau}+\sum_{\alpha,\beta=1,2}|b^\beta_\alpha|_{0,\infty,\tau}\right)\le C\eps.
\end{equation*}
This requires the mesh size along the shell free edge to be equal to the shell thickness.
One way to improve the stability in this case is localizing the penalty constant $\C$ in  \eqref{gamma-h} and
\eqref{tau-h}, and making it bigger on edges next to the shell free boundary.
\end{remark}

\begin{thm}
Let $(\!theta\e, \!u\e, w\e)$ be the solution of the shell model \eqref{N-0-model-m}.
Let $(\!theta^h, \!u^h, w^h)\in \ub\H_h$ be the solution of
the finite element model \eqref{N-0-fem-m}. 
If the triangulation $\T_h$ satisfies the condition \eqref{m-stability-condition}, then when the penalty constant $\C$ in \eqref{rho-h},
\eqref{gamma-h}, and \eqref{tau-h} is big enough, the finite element model has a unique solution and
there is a constant $C$ that only depends on the shape regularity $\K$ of the triangulation, but otherwise
is independent of the triangulation $\T_h$ and the shell thickness , such that
\begin{multline}\label{membrane-thm-est}
\eps\|(\!theta\e-\!theta^h, \!u\e-\!u^h, w\e-w^h)\|_{\H_h}\\+
\|(\!u\e-\!u^h, w\e-w^h)\|_{\gamma_h}+
\|(\!theta\e-\!theta^h, \!u\e-\!u^h, w\e-w^h)\|_{\tau_h}\\
\le C
\left\{
\sum_{\tau\in\T_h}h^2_{\tau}\left[\eps^2\sum_{\alpha=1,2}\|\theta\e_\alpha\|^2_{2,\tau}
+\sum_{\alpha=1,2}\left(\|\theta\e_\alpha\|^2_{1,\tau}+\|u\e_\alpha\|^2_{2,\tau}\right)+\|w\e\|^2_{2,\tau}\right]\right\}^{1/2}.
\end{multline}
\end{thm}
\begin{proof}
It follows from the consistency of the finite element model \eqref{N-0-fem-m} and the shell model \eqref{N-0-model-m}
and the continuity and coerciveness of the finite element model established in Lemma~\ref{m-stability-lem} that
the finite element solution has the optimal accuracy in the energy norm that
\begin{multline}
\eps\|(\!theta\e-\!theta^h, \!u\e-\!u^h, w\e-w^h)\|_{\rho_h}+
\|(\!u\e-\!u^h, w\e-w^h)\|_{\gamma_h}+
\|(\!theta\e-\!theta^h, \!u\e-\!u^h, w\e-w^h)\|_{\tau_h}\\
\lesssim
\eps\|(\!theta\e-\!theta^I, \!u\e-\!u^I, w\e-w^I)\|_{\rho_h}+
\|(\!u\e-\!u^I, w\e-w^I)\|_{\gamma_h}\\+
\|(\!theta\e-\!theta^I, \!u\e-\!u^I, w\e-w^I)\|_{\tau_h}
\ \forall\ \ (\!theta^I, \!u^I, w^I)\in \ub\H_h.
\end{multline}
We take $\theta^I_\alpha$, $u^I_\alpha$, and $w^I$ as the element-wise $L^2$ projections of $\theta\e_\alpha$, $u\e_\alpha$, and $w\e$
into the spaces of piecewise linear functions, respectively.
The results of the theorem then follows from the norm definitions \eqref{rho-h-norm}, \eqref{gamma-h-norm}, \eqref{tau-h-norm},
the Korn inequality \eqref{Hh-ah-equiv},
and the strain operator definitions \eqref{N-bending}, \eqref{N-metric}, and \eqref{N-shear}.
\end{proof}

This result implies accuracy of the finite element solution defined by the finite element model \eqref{N-0-fem-m} when the shell deformation is
membrane/shear dominated. Under the loading scale assumption made at the beginning of this section, in view of the
abstract results \eqref{abs-membrane-est}, for membrane/shear dominated shells, we have
the asymptotic estimates on the shell model solution that as $\eps\to 0$
\begin{equation*}
\begin{gathered}
(\!theta\e, \!u\e, w\e) \ \text{ converges to a finite limit in the membrane/shear energy norm},\\
\eps^2\rho(\!theta\e, \!u\e, w\e)+\gamma(\!u\e, w\e)+\tau(\!theta\e, \!u\e, w\e) \ \text{ tends to a finite non-zero limit}.
\end{gathered}
\end{equation*}
The convergence $(\!theta\e, \!u\e, w\e)$ to a finite limiting function ensures that the right hand side of
\eqref{membrane-thm-est} can be made small as long as the finite element mesh can resolve the
layers in the shell model solution. The total energy tending to a non zero limit means that the relative error in the
energy norm can be made small, which is the desired accuracy.

The estimate \eqref{membrane-thm-est} also implies accuracy of the finite element solution for intermediate shell problems.
In this case,  in view of the abstract result \eqref{abs-intermediate-est}, we have  the asymptotic behavior of the
shell model solution that
\begin{equation*}
\begin{gathered}
(\!theta\e, \!u\e, w\e) \ \text{ converges to a finite limit in an abstract space},\\
\eps^2\rho(\!theta\e, \!u\e, w\e)+\gamma(\!u\e, w\e)+\tau(\!theta\e, \!u\e, w\e)\simeq o(\eps^{-2}).
\end{gathered}
\end{equation*}
Intermediate shells usually exhibit stronger boundary and internal layers in its deformation and stress distribution. The fact that
the model solution converges to a finite limit, although in an abstract sense, still means that the upper bound in the
error estimate \eqref{membrane-thm-est} can be made small as long as the finite element  mesh can resolve
the stronger layers. Formally, the total energy tending to infinity at the rate of $o(\eps^{-2})$ seems to imply that it is easier to make the relative error
of finite element solution small. This, however,only suggests that the layers in the model solution is harder to resolve.

Although the estimate \eqref{membrane-thm-est} is  valid for bending dominated shells, it does not imply accuracy of the finite element solution.
Under the loading scale assumption of this section, it follows from the abstract results \eqref{abs-decomp} and \eqref{Thm2.1-1}
that for bending dominated shells we have the following asymptotic estimates on the shell model solution.
\begin{equation*}
\begin{gathered}
(\!theta\e, \!u\e, w\e)\simeq\eps^{-2}(\!theta^0, \!u^0, w^0),\\
\eps^2\rho(\!theta\e, \!u\e, w\e)+\gamma(\!u\e, w\e)+\tau(\!theta\e, \!u\e, w\e)\simeq\eps^{-2}.
\end{gathered}
\end{equation*}
The optimal approximation error is amplified in the finite element error by the factor $\eps^{-1}$ in the relative energy norm.
This may also be explained in terms of membrane/shear locking.

\section{A finite element procedure for shells}
\label{Procedure}
As we discussed at the end of Section~\ref{FEMmodel}, we only need to write the code for the mixed method \eqref{N-fem}
with an adjustable parameter $\vtheta$ included. The code for the discontinuous Galerkin method \eqref{N-0-fem}
can be obtained by cutting out a portion of the program.
For a given shell problem, depending on whether the shell problem is bending dominated or not,
either the mixed finite element method \eqref{N-fem} or the discontinuous Galerkin method \eqref{N-0-fem}
produces accurate numerical approximation.
It is very important to know  the asymptotic regime to which a shell problem belongs.
For many shell problems, one can determine their asymptotic  regimes {\em a priori} by partial differential equation theories,
\cite{Bathe-book, CiarletIII, Sanchez}.  We can also use numerical results
to detect the asymptotic behavior of a shell  and determine
its asymptotic regime, and thus determine which method is suitable for the given shell
problem.

The procedure is as follows. We use both the two methods to compute a given shell problem. Let $u\e_h$ represent the finite element
solution determined by  the mixed method \eqref{N-fem}, and let
$\bar u\e_h$ represent the solution of the discontinuous Galerkin method \eqref{N-0-fem}.
If $u\e_h$ is significantly bigger than $\bar u\e_h$
in the $\H_h$ norm,
the problem must be bending dominated, and the smallness
of the solution by the discontinuous Galerkin method must be caused by the numerical membrane/shear locking.
In this case, $u\e_h$ is the finite element solution one should use.
If the distinction between the two numerical solutions are not very clear, we do the following computations to amplify the distinction.

In the abstract setting, we have the orthogonal decomposition \eqref{abs-decomp} of the shell model solution that
\begin{equation*}
u\e=u^0+\eps^2 v\e, \quad u^{\eps/2}=u^0+\frac{\eps^2}{4}v^{\eps/2}.
\end{equation*}
From these we have
\begin{equation*}
\frac13(4u^{\eps/2}-u\e)=u^0+\frac{\eps^2}{3}(v^{\eps/2}-v\e).
\end{equation*}
The shell problem is bending dominated when $u^0\ne 0$. Since $v^{\eps/2}-v\e$ is convergent to zero, the three elements $u\e$,  $u^{\eps/2}$, and
$\frac13(4u^{\eps/2}-u\e)$ are successively closer to $u^0$. In the case of membrane/shear dominated or intermediate shell, we have $u^0=0$, and 
the three elements are successively closer to zero.

Similar decomposition is valid for the finite element solution determined by the mixed method \eqref{N-fem}.
Based on this, we propose to do an additional computation.
We compute  $u^{\eps/2}_h$
that represents the solution of \eqref{N-fem} with
$\eps$ being replaced by $\eps/2$. We then calculate $\frac13(4u^{\eps/2}_h-u\e_h)$.
If $u\e_h$,  $u^{\eps/2}_h$, and $\frac13(4u^{\eps/2}_h-u\e_h)$ are successively closer to a fixed non zero limit, the shell problem is bending dominated,
and $u\e_h$ is the finite element solution one should use.
If the three sets of functions are successively closer to zero, with the last one being virtually equal to zero
in most part of the shell, the shell problem is not bending dominated. In this case, one should
take $\bar u\e_h$ as the finite element solution.

\bibliographystyle{plain}

\end{document}